\documentclass[11pt,a4paper,bibliography=totocnumbered,listof=totocnumbered]{article}
\usepackage[ngerman,english]{babel}
\usepackage[T1]{fontenc}
\usepackage[utf8]{inputenc}
\usepackage{amsmath}
\usepackage{mathtools}
\usepackage{amsthm}
\usepackage{amsfonts}
\usepackage{amssymb}
\usepackage{graphicx} 
\usepackage{geometry} 
\usepackage{setspace} 
\usepackage{subfig} 

\usepackage{paralist} 
\usepackage{enumitem}

\usepackage{titlesec} 
\usepackage[colorlinks, linkcolor = blue, citecolor = blue, urlcolor = blue]{hyperref}
\usepackage[toc,page]{appendix}
\usepackage{cite}
\usepackage{bbm}
\usepackage{dsfont}

\newcommand{\mE}{\mathcal{E}}
 
\DeclareMathOperator{\Dbm}{\Delta_{\mu}^b}
\DeclareMathOperator{\mH}{\mathcal{H}}

\DeclareMathOperator{\supp}{supp}
\DeclareMathOperator{\sto}{sto}

\newtheoremstyle{normal}
{10pt}
{10pt}
{}
{}
{\bfseries}
{}
{0em}
{\bfseries{\thmname{#1}\thmnumber{ #2}\thmnote{ \hspace{0em}(#3)\newline}}}

\newtheoremstyle{standard}  
  {10pt}   
  {}   
  {\itshape}  
  {}       
  {\bfseries} 
  {:}         
  {0.2cm}  
  {}          

\newtheoremstyle{mittitel}  
  {10pt}   
  {}   
  {\itshape}  
  {}       
  {\bfseries} 
  {:}         
  {0.2cm}  
  {\bfseries{\thmname{#1}\thmnumber{ #2}\thmnote{ \hspace{0em}(#3)\newline}}}          

\title{Stochastic Heat Equations defined by Fractal Laplacians on Cantor-like Sets\vspace{-1ex} }
\author{Tim Ehnes\footnote{ Institute of Stochastics and Applications, University of Stuttgart, Pfaffenwaldring 57, 70569 Stuttgart, Germany, e-mail: tim.ehnes@mathematik.uni-stuttgart.de }}
\date{\vspace{-5ex}}

\begin{document}
\maketitle

\setlength{\parindent}{10pt}

\titlespacing{\section}{0pt}{20pt plus 4pt minus 2pt}{14pt plus 2pt minus 12pt}
\titlespacing{\subsection}{0pt}{20pt plus 4pt minus 2pt}{8pt plus 2pt minus 2pt}
\titlespacing{\subsubsection}{0pt}{12pt plus 4pt minus 2pt}{-6pt plus 2pt minus 2pt}

\theoremstyle{standard}

\newtheorem{thm}{Theorem}[section] 
\newtheorem{satz}[thm]{Proposition} 
\newtheorem{lem}[thm]{Lemma}
\newtheorem{kor}[thm]{Corollary} 
\newtheorem{defi}[thm]{Definition} 
\newtheorem{bem}[thm]{Remark}
\newtheorem{hyp}[thm]{Assumption}
\newtheorem{exa}[thm]{Example}

\begin{abstract}
We study stochastic heat equations in the sense of Walsh defined by fractal Laplacians on Cantor-like sets. For this purpose, we first investigate the corresponding heat kernels. Then, we prove existence and uniqueness of mild solutions to stochastic heat equations provided some Lipschitz and linear growth conditions. We establish Hölder continuity in space and time and compute the Hölder exponents. Moreover, we address the question of weak intermittency. 
\end{abstract}

\section{Introduction}\label{section_1}

In this paper we study parabolic stochastic partial differential equations defined by generalized second order differential operators.
To introduce the operator of interest, let $[a,b]\subset \mathbb{R}$ be a finite interval, $\mu$ a finite non-atomic Borel measure on $[a,b]$, $\mathcal{L}^2([a,b],\mu)$ the space of measurable functions $f$ such that $\int_a^b f^2d\mu<\infty$ and $L^2([a,b],\mu)$ the corresponding Hilbert space of equivalence classes with inner product $\langle f,g\rangle_{\mu}\coloneqq \int_a^b fgd\mu$. We define
\begin{align*}
\mathcal{D}_{\mu}^2\coloneqq \Big\lbrace f\in C^1[a,b]: &\exists \left(f^{\prime}\right)^{\mu}\in\mathcal{L}_2([a,b],\mu):  f^{\prime}(x)=f^{\prime}(a)+\int_a^x \left(f^{\prime}\right)^{\mu}(y)d\mu(y), ~~ x\in[a,b]\Big\rbrace.
\end{align*}
The Krein-Feller operator with respect to $\mu$ is given as 
\begin{align*}
\Delta_{\mu}: \mathcal{D}_{\mu}^2\subseteq L^2([a,b],\mu)\to L^2([a,b],\mu),~~ f\to\left(f^{\prime}\right)^{\mu}.
\end{align*}
This operator has been introduced, for example, in \cite{FAF,IKD,KS,KO,LG}, especially as the infinitesimal generator of a so-called Quasi diffusion. It is a measure-theoretic generalization of the classical second weak derivative $\Delta_{\lambda^1}$, where $\lambda^1$ is the one-dimensional Lebesgue measure. \par

In order to connect these operators with diffusion equations from a physical point of view, we consider the temperature distribution in a one-dimensional bar of length $1$ that runs from $x=0$ to $x=1$. The mass distribution of the bar shall have a density denoted by $\rho: [0,1]\to\mathbb{R}$. Further, we assume that the specific heat of the material, i.e. the amount of heat energy required to raise a mass unit by a temperature unit, is constant, as well as the thermal conductivity, which gives the ability to conduct heat. Hence, we can denote the specific heat by $c$ and the thermal conductivity by $\kappa$. Then, the temperature of the bar, the function $u(t,x)$, is determined by the heat equation 
\begin{align}
\kappa \frac{\partial^2 u}{\partial x^2}(t,x)=c\rho(x)\frac{\partial u}{\partial t}(t,x)\label{heat_equation_density}
\end{align}
with Dirichlet boundary conditions $u(t,0)=u(t,1)=0$ for all $t\geq 0$ if we assume that the temperature vanishes at the boundaries or Neumann boundary conditions 
$\frac{\partial u}{\partial x}(t,0)=\frac{\partial u}{\partial x}(t,1)=0$ if the boundaries are perfectly insulated.
In order to solve the heat equation, we use the separation of variables and write 
$u(t,x)=f(x)g(t)$, which yields
\begin{align*}
\kappa f^{\prime\prime}(x)g(t)=c\rho(x)f(x)g^{\prime}(t)
\end{align*}
and by resorting
\begin{align*}
\frac{f^{\prime\prime}(x)}{\rho(x)f(x)}=\frac{c}{\kappa}\frac{g^{\prime}(t)}{g(t)} 
\end{align*}
for all $t$ and $x$. Consequently, both sides of the equation are constant and we denote the value by $-\lambda$. We only consider the left-hand side, given by
\begin{align*}
f^{\prime\prime}(x)=-\lambda\rho(x)f(x).
\end{align*}
By integration with respect to the Lebesgue measure we get
\begin{align*}
f^{\prime}(x)-f^{\prime}(0)=-\lambda\int_0^x f(y)\rho(y)dy,
\end{align*}
which can be written as
\begin{align*}
f^{\prime}(x)-f^{\prime}(0)=-\lambda\int_0^x f(y)d\mu(y),
\end{align*}
where $\rho$ is the density of the measure $\mu$. By applying the definition of $\Delta_{\mu}$,
\begin{align*}
\Delta_{\mu} f = -\lambda f,
\end{align*}
which yields 
\begin{align*}
\kappa \Delta_{\mu}u=c\frac{\partial u}{\partial t},
\end{align*}
as a generalization of heat equation \eqref{heat_equation_density}, since this equation does not involve the density $\rho$. Consequently, we can use it to formulate the problem for measure which possess no density, in particular for singular measures.
 \par
 We are interested in the case where $\mu$ is a self-similar measure on a Cantor-like set. More precisely,
let $N\geq 2$ and $\{S_1,...,S_N\}$  be a finite family of affine contractions on $[0,1]$, i.e.
\begin{align*}
S_i: [0,1]\to[0,1], ~ S_i(x)=r_ix+b_i, ~0<r_i<1, ~ 0\leq b_i\leq 1-r_i, ~ i=1,...,N,
\end{align*} 
where  $S_1(0)=0<S_1(1)\leq S_2(0)<S_2(1)\leq ... <S_N(1)=1$.
Further, let  $\mu_1,...,\mu_N$ be weights, i.e. $\mu_1,...,\mu_N\in(0,1)$ and $\sum_{i=1}^N\mu_i=1$. It is known from \cite{HF} that a unique non-empty compact set $F\subseteq [0,1]$ exists such that 
\begin{align}
F=\bigcup_{i=1}^M S_i(F)\label{definition_set}
\end{align}
and a unique Borel probabiliy measure $\mu$ such that 
\begin{align}
\mu(A)=\sum_{i=1}^N\mu_i\mu\left(S_i^{-1}(A)\right)\label{self_similarity}
\end{align}
for any Borel set $A\subseteq [0,1]$. Further, it holds $\supp\mu=F$. The set $F$ is called Cantor-like set. \par 
The main topic of this paper is the consideration of the parabolic stochastic PDE
\begin{align}
\begin{split}
\frac{\partial}{\partial t}u(t,x)&=\Delta_{\mu}^b u(t,x)+f(t,u(t,x))+g(t,u(t,x))\xi(t,x),\\
u(0,x)&=u_0(x), \label{spde_heat_intro}
\end{split}
\end{align}
where $b\in\{N,D\}$ determines the boundary condition, $\xi$ is a space-time white noise on $L^2([0,1],\mu)$ and $f$ and $g$ are predictable processes satisfying some Lipschitz and linear growth conditions. We investigate existence and uniqueness of a mild solution to \eqref{spde_heat_intro} and, if a mild solution exists, the question whether a Hölder continuous version of this solution exists. It is known (see \cite{WI}) that the stochastic heat equation defined by the classical one-dimensional weak Laplacian $\Delta_{\lambda^1}$ has a unique mild solution which is, some regularity conditions provided, essentially $\frac{1}{2}$-Hölder continuous in space and $\frac{1}{4}$-Hölder continuous in time. Here, essentially $\alpha$-Hölder continuous means Hölder continuous for every exponent strictly less than $\alpha$. However, in two space dimensions it turns out that the mild solution is a distribution, no function (see \cite{WI}). Hambly and Yang \cite{HYC} addressed the questions regarding the above-mentioned properties in the setting of  a space with dimension between one and two, more precisely a p.c.f. self-similiar set (in the sense of \cite{KA}) with Hausdorff dimension between one and two. It turned out that, some conditions on the initial value $u_0$ and the processes $f$ and $g$ provided, there exists a version of the mild solution which is almost surely essentially $\frac{1}{2}$-Hölder continuous in space and $\frac{1}{2}(d_H+1)^{-1}$-Hölder continuous in time. Hence, the temporal Hölder exponent decreases with increasing space dimension. The Krein-Feller operator can be interpreted as a generalized Laplacian on sets with dimension less or equal one. Therefore, it seems natural to ask whether the mild solution to \eqref{spde_heat_intro} defined by $\Delta_{\mu}^b$ has, if it exists, a temporal Hölder exponent that is greater than $\frac{1}{4}$ in case of dimension less than one. Moreover, for the investigation on p.c.f. fractals in \cite{HYC}, it is assumed that the weights are given as $\mu_i=r_i^{d_{H}}$. Another aim of this paper is to find Hölder exponents for other weights. \\

Preliminary for the consideration of the mild solution to \eqref{spde_heat_intro}, we need to have a closer look on the heat kernel of $\Delta_{\mu}^b$, defined by
\begin{align*}
p_t^b(x,y)=\sum_{k=1}^{\infty} e^{-\lambda_k^b t}\varphi_k^b(x)\varphi_k^b(y),
\end{align*}
where $\lambda_k^b, k\geq 1$ are the eigenvalues and $\varphi_k^b, ~ k\geq 1$ the $L^2([0,1],\mu)$-normed eigenfunctions of the Neumann- (or Dirichlet- resp.) Krein-Feller operator $\Delta_{\mu}^b$. From \cite{ES} it is known that a constant $C_2>0$ exists such that for all $k\in\mathbb{N}$
\begin{align*}
\left\lVert \varphi_k^b\right\rVert_{\infty}\leq C_2k^{\frac{\delta}{2}},
\end{align*} 
where $\gamma$ is the spectral exponent of $\Delta_{\mu}^b$ and $\delta\coloneqq \max_{1\leq i\leq N}\frac{\log \mu_i}{\log\left((\mu_ir_i)^{\gamma}\right)}$. This estimate along with the spectral asymptotics (see \cite{FAF}) leads to an extension of the well-known heat kernel estimates (see for example \cite{KS}), which will be our main tool in the observation of the mild solution to \eqref{spde_heat_intro}. \par 
It turns out that, under some conditions on $u_0, f$ and $g$, the mild solution to \eqref{spde_heat_intro} exists, is unique and jointly continuous in space and time. Moreover, we show that, under additional conditions on $u_0$, there exists a version that is essentially $\frac{1}{2}$-Hölder continuous in space and essentially $\frac{1}{2}-\frac{\gamma\delta}{2}$-Hölder continuous in time. If $\mu$ is chosen as the $d_H$-dimensional Hausdorff measure, we obtain $\frac{1}{2}(d_H+1)^{-1}$
as essential temporal Hölder exponent, which is indeed greater than $\frac{1}{4}$ if $d_H<1$. If $\mu$ is another measure, we have $\gamma\delta>\frac{d_H}{d_H+1}$, which gives a lower temporal Hölder exponent in this case. \par
In \cite[Section 6]{HYC} the relation between 
stochastic heat equations in the sense of Walsh and stochastic PDEs on $L^2([0,1],\mu)$ in the sense of da Prato–Zabczyk (see \cite{DZS}) has been used to establish $\frac{1}{2}$ as essential spacial Hölder exponent. We do not follow this procedure. Instead, we approximate the heat kernel by proving that for $x\in F$, $t\in[0,T]$
\[\int_0^t\int_0^1\left(\left\langle p_{t-s}^b(\cdot,y),f_n^x\right\rangle_{\mu} - p_{t-s}^b(x,y)\right)^2d\mu(y)ds\to 0\] as $n\to\infty,$ where the sequence $\left(f_n^x\right)_{n\in\mathbb{N}}$ approximates the Delta functional of $x$. Then, we show that the resulting approximating mild solutions have the desired spatial continuity and that the regularity is preserved upon taking the limit. In particular, we do not need to make use of the theory of da Prato–Zabczyk. Nevertheless, stochastic heat equation in the sense of da Prato–Zabczyk defined by $\Dbm$  are, in addition to their realtion to Walsh SPEs, interesting in itself. However, we do not investigate such SPDEs in this paper. It should be noted that the investigation work very similar to the corresponding one in \cite{HYC}. \\
 
 Next to these continuity properties, we investigate the intermittency of the mild solution to \eqref{spde_heat_intro}. Roughly speaking, an intermittent process develops increasingly high peaks on small space-intervals when the time parameter  increases. This is a phenomenon of mild solution to stochastic diffusion equations which has found much attention in the last years (see, among many others, \cite{BCS}, \cite{HHS}, \cite{KAS} \cite{KKI}).  We call a mild solution $u$ weakly intermittent on $[0,1]$ if the lower and the upper moment Lyapunov exponents, which are respectively the functions $\gamma$ and $\bar \gamma$ defined by
\begin{align*}
\gamma(p,x)\coloneqq \liminf_{t\to\infty}\frac{1}{t}\log\mathbb{E}\left[u(t,x)^p\right], ~~ 
\bar\gamma(p,x)\coloneqq \limsup_{t\to\infty}\frac{1}{t}\log\mathbb{E}\left[u(t,x)^p\right], ~~ p\in(0,\infty), x\in[0,1],
\end{align*}
satisfy
\begin{align*}
\gamma(2,x)>0,~~~ \bar\gamma(p,x)<\infty, ~~ ~p\in [2,\infty), x\in[0,1],
\end{align*}
(see \cite[Definition 7.5]{KAS}).
We prove this in the Neumann case for $f=0$ and some conditions on $g$. \\

This paper is structured as follows. In Section \ref{section_2} we give definitions related to Krein-Feller operators and Cantor-like sets as well as results concerning the spectral asymptotics. Furthermore, we recall a method to approximate the resolvent density, we collect basic properties and  prove some continuity properties of heat kernels.
Section \ref{section_4}  is dedicated to the analysis of SPDE \eqref{spde_heat_intro}, including the proofs of existence, uniqueness and Hölder continuity properties of the mild solution as well as the investigation of weak intermittency.
\par

\section{Preliminaries and Preparing Estimates}\label{section_2}
\subsection{Definition of Krein-Feller Operators on Cantor-like Sets}\label{Krein-Feller Operators and Cantor-like Sets}
First, we recall the definition and some analytical properties of the operator $\Delta_{\mu}^b$, where $b\in\{N,D\}$ and $\mu$ is a self-similar measure on a Cantor-like set according to the definition in Section \ref{section_1}. \par 
We denote the support of the measure $\mu$ and thus the Cantor-like set by $F$. If $[0,1]\setminus F\neq\emptyset$, $[0,1]\setminus F$ is open in $\mathbb{R}$ and can be written as
\begin{align}
[0,1]\setminus F = \bigcup_{i=1}^{\infty} (a_i,b_i)\label{offene_mengen} 
\end{align}
with $0<a_i<b_i<1$, $a_i,b_i\in[0,1]$ for $i\geq 1$. We define 
\begin{align*}
\mathcal{D}_{\lambda^1}^1\coloneqq\left\lbrace f:[0,1]\to\mathbb{R}: \text{there exists } f^{\prime}\in \mathcal{L}^2([0,1],\lambda^1): f(x)=f(0)+\int_0^x f^{\prime}(y)d\lambda^1(y),~ x\in[0,1]  \right\rbrace
\end{align*}
and $H^1\left([0,1],\lambda^1\right)$ as the space of all $\mathcal{H}\coloneqq L^2([0,1],\mu)$-equivalence classes having a $\mathcal{D}_{\lambda^1}^1-$represen\-tative. If $\mu=\lambda^1$ on $[0,1]$, this definition is equivalent to the definition of the Sobolev space $W_2^1$.

$H^1\left([0,1],\lambda^1\right)$ is the domain of the non-negative symmetric bilinear form $\mE$ on $\mathcal{H}$ defined by
\begin{align*}
\mathcal{E}(u,v)=\int_0^1 u'(x)v'(x)dx,~~~ u,v\in \mathcal{F}\coloneqq H^1\left([0,1],\lambda^1\right).
\end{align*}
Hereby, for each argument, which is an element of $\mathcal{H}$, we choose the $\mathcal{D}_{\lambda^1}^1$-representative which is linear on $[0,1]\setminus F$. 
It is known (see \cite[Theorem 4.1]{FD}) that $\left(\mE,\mathcal{F}\right)$ defines a Dirichlet form on $\mathcal{H}$. Hence, there exists an associated non-negative, self-adjoint operator $\Delta_{\mu}^N$ on $\mathcal{H}$ with $\mathcal{F}=\mathcal{D}\left(\left(-\Delta_{\mu}^N\right)^{\frac{1}{2}}\right)$ such that
\begin{align*}
\langle -\Delta_{\mu}^N u,v\rangle_{\mu}&=\mE(u,v), ~~u\in\mathcal{D}\left(\Delta_{\mu}^N\right),v\in \mathcal{F}
\end{align*}
and it holds 
\[\mathcal{D}\left(\Delta_{\mu}^N\right)=\left\{f\in \mathcal{H}: f \text{ has a representative } \bar f \text{ with } \bar f\in \mathcal{D}_{\mu}^2 \text{ and } \bar f'(0)=\bar f'(1)=0\right\}.\] $\Delta_{\mu}^N$ is called Neumann Krein-Feller operator w.r.t. $\mu$. Furthermore, let $\mathcal{F}_0\coloneqq H^1_0\left([0,1],\lambda^1\right)$ be the space of all $\mathcal{H}$-equivalence classes which have a $\mathcal{D}_{\lambda^1}^1-$representative $f$ such that $f(0)=f(1)=0.$
The bilinear form defined by
\begin{align*}
\mathcal{E}(u,v)=\int_0^1 u'(x)v'(x)dx,~~~ u,v\in \mathcal{F}_0,
\end{align*}
is a Dirichlet form, too (see \cite[Theorem 4.1]{FD}).
Again, there exists an associated non-negative, self-adjoint operator $\Delta_{\mu}^D$ on $\mathcal{H}$ with $\mathcal{F}_0=\mathcal{D}\left(\left(-\Delta_{\mu}^D\right)^{\frac{1}{2}}\right)$ such that
\begin{align*}
\langle -\Delta_{\mu}^D u,v\rangle_{\mu}&=\mE(u,v), ~~ u\in \left(\Delta_{\mu}^D\right),~v\in \mathcal{F}_0
\end{align*}
and it holds 
\[\mathcal{D}\left(\Delta_{\mu}^D\right)=\left\{f\in \mathcal{H}: f \text{ has a representative } \bar f \text{ with } \bar f\in \mathcal{D}_{\mu}^2 \text{ and } \bar f(0)=\bar f(1)=0\right\}.\] $\Delta_{\mu}^D$ is called Dirichlet Krein-Feller operator w.r.t. $\mu$.\par

A concept to describe Cantor-like sets is given by the so-called word or code space. Let $I\coloneqq \{ 1,...,N\}$, $\mathbb{W}_n= I^n$ be the set of all sequences $\omega$ of length $|\omega|=n$, $\mathbb{W}^*\coloneqq \cup_{n\in\mathbb{N}} I^n,$ 
the set of all finite sequences and $\mathbb{W}\coloneqq I^{\infty}$ the set of all infinite sequences $\theta=\theta_1\theta_2\theta_3...$ with $\theta_i\in I$ for $i\in\mathbb{N}$. Then, $I$ is called alphabet and $\mathbb{W},~ \mathbb{W}^*,~ \mathbb{W}^n: ~ n\in\mathbb{N}$ are called word spaces. We define an ordering on $\mathbb{W}$ by denoting two words $\omega$ and $\sigma$ as equal if $\omega_i=\sigma_i$ for all $i\in\mathbb{N}$ and otherwise, we write
$\omega<\sigma :\Leftrightarrow \sigma_k<\omega_k$ or
$\omega>\sigma :\Leftrightarrow \sigma_k>\omega_k$, where $k\coloneqq\inf\{ n\in\mathbb{N}:\sigma_n\neq\omega_n\}$. In addition to an ordering we define a metric on the word space by the map $d:\mathbb{W}\times \mathbb{W}\rightarrow \mathbb{R},~ d(\omega,\sigma)=N^{-k}$ with $k$ defined as before. It is known (see e.g. \cite[Theorem 2.1]{BF})  
that for every $x\in [0,1]$ the map
\begin{align*}
\pi_x: \mathbb{W}\rightarrow F, ~~\sigma\mapsto\lim_{n \to\infty} S_{\sigma_1}\circ S_{\sigma_2}\circ...\circ S_{\sigma_n}(x) 
\end{align*}
is well-defined, continuous, surjective and independent of $x\in [0,1]$, which means $\pi_x(\sigma)=\pi_y(\sigma)$ for all $x,y\in [0,1],~ \sigma\in \mathbb{W}$. Therefore, for every $x\in[0,1]$ and every $y\in F$ there exists, at least, one element of $\mathbb{W}$ which is by $\pi_x$ associated to $y$.

\subsection{Spectral Theory of Krein-Feller Operators}
Let $b\in\{N,D\}$ and let $\mu$ be a self-similar measure on a Cantor-like set according to the given conditions.
Further, let $\gamma$ be the spectral exponent of $-\Delta_{\mu}^b$, that is the unique solution of
\begin{align}
\sum_{i=1}^N(\mu_ir_i)^{\gamma}=1. \label{spectral_exponent}
\end{align}
It is known from \cite[Proposition 6.3, Lemma 6.7, Corollary 6.9]{FA} that there exists an orthonormal basis $\lbrace \varphi_k^b: k\in\mathbb{N}\rbrace$ of $\mathcal{H}$ consisting of $\mathcal{H}$-normed eigenfunctions of $-\Delta_{\mu}^b$ and that for the related ascending ordered eigenvalues $\{\lambda^{b}_i:i\in\mathbb{N}\}$ it holds $0\leq\lambda_1^b\leq\lambda_2^b\leq...,$ where $\lambda_1^D>0$. 
Furthermore, by \cite{FAF} there exist constants $C_0,C_1>0$ such that for $k\geq 2$ 
\begin{align}
C_0k^{\frac{1}{\gamma}}\leq \lambda_k^b\leq C_1k^{\frac{1}{\gamma}}. \label{eigenwertabsch}
\end{align}

Further, we have the following upper bound for the uniform norm of the eigenfunctions.
\begin{satz}\label{eigenfunction_estimate_theorem}
Let $\delta\coloneqq \max_{1\leq i\leq N}\frac{\log \mu_i}{\log\left((\mu_ir_i)^{\gamma}\right)}$. Then, there exists a constant $\bar C_2>0$ such that for all $k\in\mathbb{N}$ 
\begin{align*}
\left\lVert\varphi_k^b\right\rVert_{\infty}\leq \bar{C_2} \left(\lambda_k^b\right)^{\frac{\gamma}{2}\delta}.
\end{align*}
\end{satz}
\begin{proof}
\cite[Theorem 2.1]{ES}
\end{proof}
Hereby, $\left\lVert f\right\Vert_{\infty}\coloneqq \text{ess sup}_{x\in[0,1]}|f(x)|$. This is an estimate for the essential supremum, but it also holds for the supremum of the representative in $\mathcal{D}^2_{\mu}$, since this representative is continuous on $[0,1]$ and linear on $(a_i,b_i),$ $i\in\mathbb{N}$,  Inequality \eqref{eigenwertabsch} implies with $C_2\coloneqq C_1^{\frac{\delta}{2}}\bar{C_2}$
\begin{align}
\left\lVert\varphi_k^b\right\rVert_{\infty}\leq C_2 k^{\frac{\delta}{2}}.\label{eigenfunction_estimate}
\end{align}
Note that $C_2$ does not depend on $k$, but may depend on the underlying measure $\mu$.

\subsection{Properties of the Resolvent Operator}\label{resolvent_chapter}

For $\lambda>0$ and $b\in\{N,D\}$ let $\rho_{\lambda}^b$ be the resolvent density of $\Delta_{\mu}^b$. That is, with $R^{\lambda}_{b}\coloneqq (\lambda-\Delta_{\mu}^b)^{-1}$ it holds
\begin{align*}
R^{\lambda}_{b}f(x)=\int_0^1\rho_{\lambda}^b(x,y)f(y)d\mu(y), ~~~ f\in \mH.
\end{align*}
Such a mapping exists and is given by (compare \cite[Theorem 6.1]{FA})
\begin{align*}
\rho_{\lambda}^N(x,y) = \rho_{\lambda}^N(y,x) = \left(B^{\lambda}_N\right)^{-1} g_{1,N}^{\lambda}(x)g_{2,N}^{\lambda}(y), ~~ x,y\in[0,1], x\leq y, \\
\rho_{\lambda}^D(x,y) = \rho_{\lambda}^D(y,x) = \left(B^{\lambda}_D\right)^{-1}g_{1,D}^{\lambda}(x)g_{2,D}^{\lambda}(y), ~~ x,y\in[0,1], x\leq y, 
\end{align*}
where $B^{\lambda}_N,B^{\lambda}_D$ are non-vanishing constants and  the mappings $g_{1,N}^{\lambda}, g_{2,N}^{\lambda}, g_{1,D}^{\lambda}, g_{2,D}^{\lambda}$ are eigenfunctions of $\Delta_{\mu}$ with appropriate boundary conditions (see \cite[Remark 5.2]{FA}). Moreover, it is Lipschitz continuous.

\begin{satz}\label{resolvent_density_lipschitz}
Let $\lambda>0$. Then, for every $\lambda>0$ there exists a constant $L_{\lambda}\geq 0$ such that
\begin{align*}
\left| \rho_{\lambda}^b(x,y)-\rho_{\lambda}^b(x,z)\right|\leq L_{\lambda}|y-z|, ~~~ x,y,z\in[0,1].
\end{align*}
\end{satz}
\begin{proof}
\cite[Proposition 2.6]{ES}
\end{proof}

The following result connects the introduced resolvent and the semigroup associated to $\Delta_{\mu}^b$.
\begin{lem} \label{resolvent_semigroup_lemma}
Let $\lambda>0, f\in\mH$. Then,
\begin{align*}
R^{\lambda}_b f=\int_0^{\infty}e^{-\lambda t}T_t^bfdt.
\end{align*}
\end{lem}
\begin{proof}
\cite[Theorem 1.10]{ENO}
\end{proof}

\subsection{Approximation of the Resolvent Density}
We recall a method to approximate the delta functional on Cantor-like sets, in particular to approximate the just introduced resolvent density, which will then again  be used to approximate point evaluations of heat kernels.\par 
For $n\geq 1$ let $\Lambda_n$ be the partition of the word space $\mathbb{W}$ be defined by
\begin{align*}
\Lambda_n=\{\omega=\omega_1...\omega_m\in \mathbb{W}^*:
 r_{\omega_1}\cdots r_{\omega_{m-1}}>r_{\max}^{n}\geq r_{\omega}\},
\end{align*}
where $r_{\max}\coloneqq\max_{i=1,...,N}r_i$.
Moreover, let $\nu_i=\frac{\mu_i}{r_i^{d_H}}, 1\leq i\leq N$, where $d_H$ is the Hausdorff dimension of $F$. Further, for $\omega\in\mathbb{W}$ we denote $S_{\omega}(F)$ by $F_{\omega}$.
\begin{lem}\label{partitionslemma}
It holds for $n\in\mathbb{N}$:
\begin{enumerate}[label=(\roman*)]
\item \label{partitionslemma_1} $|\Lambda_n|<\infty$ and $\bigcup_{\omega\in \Lambda_n}F_{\omega}=F.$
\item \label{partitionslemma_2} For $\omega\in \Lambda_n$ there exists a subset $\Lambda'\subseteq\Lambda_{n+1}$ such that $F_{\omega}=\bigcup_{\nu\in\Lambda'}F_{\nu}.$
\item \label{partitionslemma_3} For $\omega,\nu\in\Lambda_n$, $\omega\neq\nu$ it holds $|F_{\omega}\cap F_{\nu}|\leq 1.$
\item \label{partitionslemma_estimate} For $\omega\in\Lambda_n$ it holds $\mu(F_{\omega})>r_{\max}^{nd_H}r_{\min}^{d_H}\nu_{\min}^n$.
\item \label{partitionslemma_5} For $w\in\mathbb{W}^*$ there exists $n\in\mathbb{N}$ such that $w\in\Lambda_n$. Consequently, for all $m\geq n$ there exists $\Lambda'_m\subseteq \Lambda_m$ such that $F_w=\cup_{\nu\in\Lambda'_m}F_{\nu}$.
\end{enumerate}
\end{lem}

If the measure $\mu$ is chosen as $\mu_i=r_i^{d_H}$ and thus $\nu_i = 1,~ i=1,...,N$, we get an estimate similar to \cite[Lemma 3.5(iv)]{HYC}. Note that these ideas can be used to generalize the corresponding results in \cite{HYC}.
\begin{proof}
\cite[Lemma 2.7]{ES}
\end{proof}

We introduce a sequence of functions approximating the Delta functional. Hereby, we use the notation of \cite{HYC}.  We prepare this definition by defining the $n$-neighbourhood of $x\in F$ for $n\in\mathbb{N}$ by
\begin{align*}
D^0_n(x)\coloneqq \bigcup\{F_w:w\in\Lambda_n,~ x\in F_w\}.
\end{align*}
Note that $D_n^0(x)$ consists of at least one element of $\{F_w, w\in\Lambda_n\}$, which follows from Lemma \ref{partitionslemma}\ref{partitionslemma_1}, and of at most two elements since pairs of these elements intersect in at most one point. From the latter and the definition of $\Lambda_n$ it follows 
\begin{align}
|D_n^0(x)|\leq 2r_{\max}^n.\label{Dn0_diameter}
\end{align}
 With that, we can define the approximating functions for $x\in F$ and $n\geq 1$ by
\begin{align*}
f_n^x=\mu(D_n^0(x))^{-1}\mathds{1}_{D_n^0(x)}.
\end{align*}
From Lemma \ref{partitionslemma}\ref{partitionslemma_estimate} it follows
\begin{align}
||f_n^x||^2_{\mu}=\mu(D_n^0(x))^{-1}\leq r_{\max}^{-nd_H}r_{\min}^{-d_H}\nu_{\min}^{-n}.\label{f_n^x_norm_esimate}
\end{align}
We deduce the following result.
\begin{lem}\label{approximation_result}
Let $x\in F$. It holds $\lim_{n\to\infty}\langle f_n^x,g\rangle_{\mu}=g(x)$ for any continuous $g\in\mH$.
\end{lem}
\begin{proof}
\cite[Lemma 2.8]{ES}
\end{proof}

\begin{lem}\label{resolvent_estimate}
Let $x_1,x_2\in F$ and $m,n\geq 1.$ Then,
\begin{align*}
\left|\int_0^1\int_0^1\rho_1^b(y,z)f_m^{x_1}(y)f_n^{x_2}(z)d\mu(y)d\mu(z)-\rho_1^b(x_1,x_2)\right|\leq2 L_1(r_{\max}^n+r_{\max}^m),
\end{align*}
where $L_1$ denotes the Lipschitz constant of $\rho_1^b$. 
\end{lem}
\begin{proof}
\cite[Lemma 2.9]{ES}
\end{proof}

\subsection{Heat Kernel Properties}\label{heat_kernels_chapter}
Preliminary for the definition of a mild solution of a heat equation defined by a white noise integral, we introduce the notion of a heat kernel. For that, let $b\in\{N,D\}$ be fixed.
\begin{defi}
 For $(t,x,y)\in(0,\infty)\times[0,1]\times[0,1]$ define
\begin{align*}
p_t^b(x,y)\coloneqq \sum_{k=1}^{\infty} e^{-\lambda_k^b t}\varphi_k^b(x)\varphi_k^b(y).
\end{align*}
This is called \textbf{heat kernel} of $\Delta_{\mu}^b$.
\end{defi}
Moreover, we define for $h\in \mathcal{H}$ $\int_0^1p_0^b(x,y)h(y)d\mu(y)=h(x)$. By part \ref{heat_kernel_3} of the next proposition this is a meaningful definition.\par 

\begin{satz}\label{heat_kernel_satz} Let $T>0$, $h\in \mathcal{H}$ and $(T_t^b)_{t\geq 0}$ be the transition semigroup associated to $\Delta_{\mu}^b$. 
\begin{enumerate}[label=(\roman*)]
\item \label{heat_kernel_0} There exists $K_T>0$ such that $\left|p_t^b(x,y)\right|<K_T$ for all $(t,x,y)\in [T,\infty)\times[0,1]^2$.
\item \label{heat_kernel_1} $p_{\cdot}(\cdot,\cdot)$ is continuous on $(0,\infty)\times[0,1]^2$.
\item \label{heat_kernel_2} $T_t^bh(x)=\left\langle p_t(x,\cdot),h(\cdot)\right\rangle_{\mu}$ in $L^2(\mu)$ for $t\in(0,\infty)$.
\item \label{heat_kernel_3} $\int_0^1 p_s(x,z)p_t(z,y)d\mu(z)=p_{t+s}(x,y)$ for all $t>0,s\geq 0, x,y\in[0,1]$.
\item \label{heat_kernel_6} For $(t,x,y)\in [0,\infty)\times[0,1]^2$ let $p_{t,x}^b(y)\coloneqq p_t^b(x,y)$. Then, $p_{t,x}^b\in\mathcal{D}\left(\Delta_{\mu}^b\right)$ and it holds 
\begin{align*}
\frac{\partial}{\partial t}p_t^b(x,y)=-\Delta_{\mu}^b p^b_{t,x}(y)
\end{align*}
for all $t\in(0,\infty)$, $x,y\in[0,1]$.
\item\label{heat_kernel_4} $p_t(x,y)\geq 0$ for all $(t,x,y)\in(0,\infty)\times[0,1]^2$.
\item \label{heat_kernel_7} $\int_0^1 p_t^N(\cdot,y)d\mu(y) = 1$ and $\int_0^1 p_t^D(\cdot,y)d\mu(y) \leq 1$ for all $t\in(0,\infty).$
\item \label{heat_kernel_5} $\sup_{x,y\in[0,1]}\left|p_t(x,y)\right|=\lVert T_t\rVert_{1\to\infty}$, where $\lVert A\rVert_{p\to q}$ denotes the operator norm of an operator $A:L^p\to L^q$.

\end{enumerate}
\end{satz}
\begin{proof}
 \ref{heat_kernel_0}-\ref{heat_kernel_7} are well-known (see e.g. \cite{KS}).  The proof of \ref{heat_kernel_5} is a standard argument.
Let $t\in(0,\infty)$ be fixed and $K\coloneqq\sup_{x,y\in[0,1]}\left|p_t^b(x,y)\right|.$ Further, let $f\in L^1(\mu)$. Then it holds for $x\in[0,1]$
\begin{align*}
\left|T_t^bf(x)\right|=\left|\int_0^1p_t(x,y)f(y)d\mu(y)\right|\leq K\int_0^1\left|f(y)\right|d\mu(y)
\end{align*}
and thus $\left\lVert T_t^b\right\rVert_{1\to\infty}\leq K$.\par
Since $p_t$ is continuous on $[0,1]^2$, there exists $x_0,y_0\in[0,1]$ such that $p_t(x_0,y_0)=K$. Define $f_n^{x_0}(x)=\frac{1}{\mu(D_n^0(x_0))}\mathbbm{1}_{D_n^0(x_0)}(x)$. We have $\left\lVert f_n^{x_0}\right\rVert_1=1$. By Lemma \ref{approximation_result}
\begin{align*}
\lim_{n\to\infty}\left\langle f_n^{x_0}(\cdot),p_t^b(\cdot,y_0)\right\rangle_{\mu} = p_t^b(x_0,y_0)=K.
\end{align*}
Hence, for all $\varepsilon>0$ there exists $n\in\mathbb{N}$ such that
\begin{align*}
\left|T_t^bf_n^{x_0}(y_0)\right|&=\frac{1}{\mu(D_n^0(x_0))}\int_{D_n^0(x_0)}p_t^b(x,y_0)d\mu(x)\\
&=\left\langle f_n^{x_0}(\cdot),p_t^b(\cdot,y_0)\right\rangle_{\mu}\\
&\geq K-\varepsilon.
\end{align*}
It follows $\left\lVert T_t^b f_n^{x_0}\right\rVert_{\infty}\geq K-\varepsilon$, which implies $\left\lVert T_t\right\rVert_{1\to\infty}\geq K$ since $\varepsilon$ can be chosen arbitrarily small. 

\end{proof}

%

An important part of the analysis of heat equations is given by estimating heat kernels. In the following proposition we prove heat kernel properties which are similar to properties on connected p.c.f. fractals established in \cite[Lemma 6.6]{HYC}, but concern all measures according to the assumptions in Section \ref{section_1} instead of only the Hausdorff measure.

\begin{satz}\label{heat_kernel_estimates}
Let $T>0$.
\begin{enumerate}[label=(\roman*)]
\item \label{heat_kernel_estimate_1} There exists $C_5(T)>0$ such that for all $(t,x,y)\in(0,T]\times[0,1]^2$
\begin{align*}
p_t^b(x,y)\leq C_5(T)t^{-\gamma\delta}.
\end{align*}
\item \label{heat_kernel_estimate_2} There exists $C_6(T)>0$ such that for all $(t,x,x',y)\in(0,T]\times[0,1]^3$
\begin{align*}
\left|p_t(x,y)-p_t(x',y)\right|\leq C_6(T)|x-x'|^{\frac{1}{2}} t^{-\frac{1}{2}-\frac{\gamma\delta}{2}}.
\end{align*}
\item \label{heat_kernel_estimate_3} There exists $C_7(T)>0$ such that for all $(s,t,x)\in(0,T]^2\times[0,1]$ with $s\leq t$
\begin{align*}
\left|p_s(x,x)-p_t^b(x,x)\right|\leq C_7(T)\left(s^{-\gamma\delta}-t^{-\gamma\delta}\right).
\end{align*}
\end{enumerate}
\end{satz}
\begin{proof}
\begin{enumerate}[label=(\roman*)]
\item We can use \cite[Proposition B.3.7]{KA} as in \cite[Theorem 2.1]{ES} to obtain the existence of a constant $C_5(1)$ such that $\left\lVert T_t^b\right\rVert_{1\to\infty}\leq C_5(1)t^{-\gamma\delta}$ for $t\in(0,1]$. With Proposition \ref{heat_kernel_satz} \ref{heat_kernel_5} it follows
\begin{align*}
\sup_{x,y\in[0,1]}p_t(x,y)\leq C_5(1)t^{-\gamma\delta}, ~~ t\in(0,1].
\end{align*}
If $T>1,$ the assertion follows from the previous inequality and the fact that $p^b$ is continuous and thus bounded on $[1,T]\times[0,1]^2$.
\item Let $t\in[0,T], y\in[0,1]$. By using \ref{heat_kernel_estimate_1} we get for $x\in[0,1]$
\begin{align*}
\left\lVert p^b_{\frac{t}{2}}(x,\cdot)\right\rVert_{\mu}^2=\int_0^1p^b_{\frac{t}{2}}(x,y)^2d\mu(y)=p_{t}(x,x)\leq C_5(T)t^{-\gamma\delta}.
\end{align*}
We set $u=p^b_{\frac{t}{2}}(\cdot,y)$. With \cite[Lemma 1.3.3(i)]{FOD}, the contractivity of $S_{\frac{t}{2}}^b$ and the above inequality it follows
\begin{align*}
\mE\left(S^b_{\frac{t}{2}}u,S^b_{\frac{t}{2}}u\right)&\leq \frac{1}{t}\left(\lVert u\rVert^2-\left\lVert S^b_{\frac{t}{2}} u\right\rVert^2\right)\\
&\leq \frac{1}{t}\lVert u\rVert^2\\
&\leq 2C_5(T)t^{-1-\gamma\delta}.
\end{align*}
Since it holds $p^b_t(\cdot,y)\in\mathcal{D}\left(\Delta_{\mu}^b\right)$ we can use the Cauchy-Schwarz inequality to get for $x,x'\in[0,1]$
\begin{align*}
\left|p_t^b(x,y)-p_t^b(x',y)\right|\leq \int_{x}^{x'}\left|\frac{\partial}{\partial z}p_t^b(z,y)\right|dz\leq |x-x'|^{\frac{1}{2}}\left(\int_{0}^{1}\left|\frac{\partial}{\partial z}p_t^b(z,y)\right|^2dz\right)^{\frac{1}{2}}.
\end{align*}
We obtain
\begin{align*}
\left|p_t^b(x,y)-p_t^b(x',y)\right|^2&\leq |x-x'|\mE\left(p_t^b(\cdot,y),p_t^b(\cdot,y)\right)\\&=|x-x'|\mE\left(S^b_{\frac{t}{2}}u,S^b_{\frac{t}{2}}u \right)\\&\leq \left|x-x'\right|2C_5(T)t^{-1-\gamma\delta}
\end{align*}
and therefore
\begin{align*}
\sup_{x,x'\in[0,1]}\frac{\left|p_t^b(x,y)-p_t^b(x',y)\right|}{\left|x-x'\right|^{\frac{1}{2}}}\leq \sqrt{2C_5(T)}t^{-\frac{1}{2}-\frac{\gamma\delta}{2}}.
\end{align*}

\item We have
\begin{align*}
\frac{\partial}{\partial t}p_t^b(x,x)&=\frac{\partial}{\partial t}\int_0^1 p^b_{\frac{t}{2}} (x,y)^2d\mu(y)\\
&= \int_0^1 \frac{\partial}{\partial t}p^b_{\frac{t}{2}} (x,y)^2d\mu(y)\\
&= 2 \int_0^1 p^b_{\frac{t}{2},x}(y)\frac{\partial}{\partial t}p^b_{\frac{t}{2},x}(y)d\mu(y)\\
&= -2 \int_0^1 p^b_{\frac{t}{2},x}(y)\Delta_{\mu}^b p^b_{\frac{t}{2},x}(y)d\mu(y)\\
&= -2 \mE\left(p^b_{\frac{t}{2},x},p^b_{\frac{t}{2},x}   \right)\leq 0,
\end{align*}
where we can interchange integral and derivative since $\frac{\partial}{\partial t}p^b_{\frac{t}{2}} (x,y)^2$ is bounded on $[t-\varepsilon,t+\varepsilon]\times[0,1]^2$ for an appropriate $\varepsilon>0$. From the previous identities and the proof of part \ref{heat_kernel_estimate_2}  we also get that there exists a $C_7'(T)$ such that
\begin{align*}
\left|\frac{\partial}{\partial t}p_t^b(x,x)\right|=\left|2\mE\left(p^b_{\frac{t}{2},x}\right)\right|\leq C_7'(T)t^{-1-\gamma\delta},
\end{align*}
where the last inequality follows as in the proof of \ref{heat_kernel_estimate_2}. Therefore
\begin{align*}
2^{\gamma} C_7'(T)t^{-1-\gamma\delta}  \leq \frac{\partial}{\partial t} p_t^b(x,x)\leq 0
\end{align*}
and we can conclude that for all $x\in[0,1]$
\begin{align*}
\left|p_s^b(x,x)-p^b_t(x,x)\right|\leq C_7'(T)\int_s^t z^{-1-\gamma\delta}dz=C_7'(T)\left(s^{-\gamma\delta}-t^{-\gamma\delta}\right).
\end{align*}
\end{enumerate}
\end{proof}

\section{Analysis of Stochastic Heat Equations}\label{section_4}
\subsection{Preliminaries}

Let $(\Omega,\mathcal{F},\mathbb{F}\coloneqq (\mathcal{F}_t)_{t\geq 0},\mathbb{P})$ be a filtered probability space statisfying the usual conditions. The object of study in this section is the stochastic PDE
\begin{align}
\begin{split}
\frac{\partial}{\partial t}u(t,x)&=\Delta_{\mu}^b u(t,x)+f(t,u(t,x))+g(t,u(t,x))\xi(t,x)\label{spde_heat}\\
u(0,x)&=u_0(x)
\end{split}
\end{align}
for $(t,x)\in[0,T]\times [0,1]$, where $T>0$, $b\in \{N,D\}$, $u_0:\Omega\times[0,1]\to\mathbb{R}$, $f,g:\Omega\times[0,T]\times[0,1]\to\mathbb{R}$. Further, $\xi$ denotes a $\mathbb{F}$-space-time white noise on $([0,1],\mu)$, that is a mean-zero set-indexed Gaussian process on $\mathcal{B}\left([0,T]\times[0,1]\right)$ such that $\mathbb{E}\left[\xi(A)\xi(B)\right]=\left|A\cap B\right|$ (compare \cite[Chapter 1]{WI}). Moreover, let for a time interval $I\subseteq[0,T]$ and a space interval $J\subseteq[0,\infty)$ $\mathcal{P}_{I,J}$ be the $\sigma$-algebra generated by simple functions on $\Omega\times I\times J$, where a simple function on $\Omega\times I \times J$ is defined as a finite sum of functions $h:\Omega\times I\times J\to\mathbb{R}$ of the form
\begin{align*}
h(\omega,t,x)=X(\omega)\mathds{1}_{(a,b]}(t)\mathds{1}_{B}(x), ~~ (\omega,t,x)\in \Omega\times I\times J
\end{align*}
with $X$ bounded and $\mathcal{F}_a$-measurable, $a,b\in I,a<b$ and $B\in\mathcal{B}(J)$.
\begin{defi}
\begin{enumerate}[label=(\roman*)]
\item
Let $q\geq 2,~ T>0$ be fixed. Let $S_{q,T}$ be the space of $[0,T]\times[0,1]$-indexed processes $v$ which are predictable (i.e. measurable from $\mathcal{P}_{[0,T],[0,1]}$ to $\mathcal{B}(\mathbb{R}))$ and for which it holds
\begin{align*}
\left\lVert v\right\rVert_{q,T}\coloneqq \sup_{t\in[0,T]}\sup_{x\in[0,1]}\left(\mathbb{E}|v(t,x)|^q\right)^{\frac{1}{q}}<\infty.
\end{align*}
Furthermore, define $\mathcal{S}_{q,T}$ as the space of equivalence classes of processes in $S_{q,T}$, where two processes $v_1,v_2$ are equivalent if $v_1(t,x)=v_2(t,x)$ almost surely for all $(t,x)\in[0,T]\times[0,1]$. 
\item For processes which are not time-dependent, we define analogously $S_q$ as the space of $[0,1]$-indexed processes $v$ which are measurable from $\mathcal{F}_0\otimes\mathcal{B}(F)$ into $\mathcal{B}(\mathbb{R})$ and which satisfy 
\begin{align*}
\left\lVert v\right\rVert_{q}\coloneqq\sup_{x\in[0,1]}\left(\mathbb{E}|v(x)|^q\right)^{\frac{1}{q}}<\infty.
\end{align*}
and the space $\mathcal{S}_q$ by identifying processes $v_1$ and $v_2$ for which it holds $v_1(x)=v_2(x)$ almost surely for all $x\in[0,1]$. 
\end{enumerate}
\end{defi}
Note that $\mathcal{S}_{q}$ and $\mathcal{S}_{q,T}$ are Banach spaces. The proof works by using standard arguments, so we skip it here.

\subsection{Existence, Uniqueness and Hölder Continuity}
Let $b\in\{N,D\}$ and $T>0$ be fixed for this chapter. We define the concept of a solution to \eqref{spde_heat} which we observe in this paper.
\begin{defi}
A \textbf{mild solution} to the SPDE \eqref{spde_heat} is defined as a predictable $[0,T]\times[0,1]$-indexed process such that for every $(t,x)\in[0,T]\times[0,1]$ it holds almost surely
\begin{align}
\begin{split}
u(t,x)=&\int_0^1p_t^b(x,y)u_0(y)d\mu(y)+\int_0^t\int_0^1p_{t-s}^b(x,y)f(s,u(s,y))d\mu(y)ds\\
&+\int_0^t\int_0^1p_{t-s}^b(x,y)g(s,u(s,y))\xi(s,y)d\mu(y)ds,\label{mild_solution_heat}
\end{split}
\end{align}
where the last term is a stochastic integral in the sense of \cite[Chapter 2]{WI}.
\end{defi}

\par
In this chapter we assume the following, which is adapted from \cite[Hypothesis 6.2]{HYC}.

\begin{hyp}\label{hypo_heat} There exists $q\geq 2$ such that
\begin{enumerate}[label=(\roman*)]
\item \label{hypo_heat_i} $u_0\in \mathcal{S}_q$
\item \label{hypo_heat_ii} $f$ and $g$ are predictable and satisfy the following Lipschitz and linear growth conditions: There exists $L>0$ and a real predictable process $M:\Omega\times[0,T]\to\mathbb{R}$ with \linebreak $\sup_{s\in[0,T]}\lVert M(s)\rVert_{L^q(\Omega)}$ such that for all $(w,t,x,y)\in\Omega\times[0,T]\times\mathbb{R}^2$
\begin{align*}
|f(\omega,t,x)-f(\omega,t,y)|+|g(\omega,t,x)-g(\omega,t,y)|&\leq L|x-y|,\\ |f(\omega,t,x)|+|g(\omega,t,x)|&\leq M(w,t)+L|x|.
\end{align*} 
\end{enumerate}
\end{hyp}

We need some preparing lemmas before proving stochastic continuity results. The following lemma shows how to find upper estimates of functionals of the heat kernel by using the resolvent density.

\begin{lem}\label{estimate_resolvent_lemma}
Let $g\in\mathcal{H}$ and $t\in(0,\infty)$. Then,
\begin{align*}
\int_0^t\int_0^1\left(\int_0^1 p_{s}^b(x,y)h(y)d\mu(y)\right)^2d\mu(x)ds\leq \frac{e^{2t}}{2}\int_0^t\int_0^1 \rho_1(x,y)h(x)g(y)d\mu(x)d\mu(y).
\end{align*}
\begin{proof}
Let $g=\sum_{k=1}^{\infty}g_k\varphi_k^b$. We adapt ideas from \cite[Lemma 4.6]{HYE}. By using Lemma \ref{heat_kernel_satz}\ref{heat_kernel_2}, the self-adjointness of the semigroup $\left(T_t^b\right)_{t\geq 0}$ and Lemma  \ref{resolvent_semigroup_lemma},
\begin{align}
\int_0^t\int_0^1\left(\int_0^1 p_{s}^b(x,y)h(y)d\mu(y)\right)^2d\mu(x)ds&=\int_0^t\left\lVert S_{s}^b h \right\rVert^2_{\mu}ds\notag\\
&= \int_0^t\left\langle S_{s}^b h, S_{s}^b h \right\rangle_{\mu} ds\notag\\
&= \int_0^t\left\langle S_{2s}^b h, h \right\rangle_{\mu} ds\notag\\
&\leq e^{2t}\int_0^te^{-2s}\left\langle S_{2s}^b h, h \right\rangle_{\mu} ds\notag\\
&\leq e^{2t}\left\langle \int_0^{\infty} e^{-2s}S_{2s}^b h ds, h \right\rangle_{\mu} \notag\\
&= \frac{e^{2t}}{2}\left\langle\int_0^1 \rho_1^b(\cdot,y)h(y)d\mu(y), h \right\rangle_{\mu} \notag\\
&= \frac{e^{2t}}{2}\int_0^1\int_0^1 \rho_1^b(x,y)h(x)h(y)d\mu(x)d\mu(y). \notag
\end{align}
\end{proof}
\end{lem}
This leads to a useful approximation of $p^b_{\cdot}(x,\cdot)$ for fixed $x\in F$.
\begin{lem} \label{estimate_resolvent_lemma_2}
Let $t\in(0,\infty)$ and $x\in F$.  Then,
\begin{align*}
\int_0^t\int_0^1\left(\left\langle p_{t-s}^b(\cdot,y),f_n^x\right\rangle_{\mu} - p_{t-s}^b(x,y)\right)^2d\mu(y)ds\leq 4L_1e^{2t}r_{\max}^{n}.
\end{align*}
\end{lem}
\begin{proof}
Let $x\in F$. In preparation for the proof we calculate
\begin{align}
&\left|\int_0^1\int_0^1\rho_1^b(z,y)(f_m^x(z)-f_n^x(z))(f_m^x(y)-f_n^x(y))d\mu(z)d\mu(y)\right|\notag\\
&= \left|\int_0^1\int_0^1\rho_1^b(z,y)\left(f_m^x(z)f_m^x(y)-f_m^x(z)f_n^x(y)-f_n^x(z)f_m^x(y)+f_n^x(z)f_n^x(y)\right)d\mu(z)d\mu(y)\right|\notag\\
&= \Big|\int_0^1\int_0^1\rho_1^b(z,y)f_m^x(z)f_m^x(y)-\rho_1^b(x,x)-\rho_1^b(z,y)f_m^x(z)f_n^x(y)+\rho_1^b(x,x)\notag\\&~~~-\rho_1^b(z,y)f_n^x(z)f_m^x(y)+\rho_1^b(x,x)+\rho_1^b(z,y)f_n^x(z)f_n^x(y)-\rho_1^b(x,x)d\mu(z)d\mu(y)\Big|\notag\\
&\leq 8L_1(r_{\max}^m+r_{\max}^n)\label{resolvent_approximation_estimate},
\end{align}
where we have used Lemma \ref{resolvent_estimate} in the last step.
Further, we note that, since for any $(t,y)\in[0,\infty)\times[0,1]$ $p_t^b(\cdot,y)$ is an element of $\mathcal{H}$ and the inner product is continuous   in each argument, it holds for any $h\in\mathcal{H}$
\begin{align}
\left\langle p_t^b(\cdot,y),h\right\rangle_{\mu}&=
\left\langle \sum_{k=1}^{\infty}e^{-\lambda_k^bt}\varphi_k^b(y)\varphi_k^b,g\right\rangle_{\mu}\notag\\
&=\left\langle \lim_{m\to\infty}\sum_{k=1}^{m}e^{-\lambda_k^bt}\varphi_k^b(y)\varphi_k^b,g\right\rangle_{\mu}\notag\\
&= \sum_{k=1}^{\infty}e^{-\lambda_k^bt}\varphi_k^b(y)\left\langle\varphi_k^b,g\right\rangle_{\mu}.\label{inner_product_cont}
\end{align}
Now, let $s,t\in[0,\infty), s<t$ and $x\in F$. By Lemma \ref{approximation_result} and Fatou's Lemma,
\begin{align*}
&\int_0^1\left( \left\langle p_{t-s}^b(\cdot,y),f_n^x\right\rangle_{\mu} - p_{t-s}^b(x,y)\right)^2d\mu(y) \\ &=
\int_0^1 \left(\sum_{k=1}^{\infty}e^{-\lambda_k^b(t-s)}\varphi_k^b(y)\left\langle \varphi_k^b,f_n^x\right\rangle_{\mu} - p_{t-s}^b(x,y)\right)^2d\mu(y)\\
&=\int_0^1 \left(\sum_{k=1}^{\infty}e^{-\lambda_k^b(t-s)}\left[\left\langle \varphi_k^b,f_n^x\right\rangle_{\mu} - \varphi_k^b(x)\right]\varphi_k^b(y)\right)^2d\mu(y)\\
&= \sum_{k=1}^{\infty}e^{-2\lambda_k^b(t-s)}\left[\left\langle \varphi_k^b,f_n^x\right\rangle_{\mu} - \varphi_k^b(x)\right]^2\\
&= \sum_{k=1}^{\infty}e^{-2\lambda_k^b(t-s)}\left[\left\langle \varphi_k^b,f_n^x\right\rangle_{\mu} - \lim_{m\to\infty}\left\langle\varphi_k^b,f_m^x\right\rangle_{\mu}\right]^2\\
&= \sum_{k=1}^{\infty}\lim_{m\to\infty}e^{-2\lambda_k^b(t-s)}\left[\left\langle \varphi_k^b,f_n^x\right\rangle_{\mu} - \left\langle\varphi_k^b,f_m^x\right\rangle_{\mu}\right]^2\\
&\leq \liminf_{m\to\infty}\sum_{k=1}^{\infty}e^{-2\lambda_k^b(t-s)}\left[\left\langle \varphi_k^b,f_n^x\right\rangle_{\mu} - \left\langle\varphi_k^b,f_m^x\right\rangle_{\mu}\right]^2.
\end{align*}
Again by Fatou's Lemma and \eqref{inner_product_cont}
\begin{align*}
&\int_0^t\int_0^1\left( \left\langle p_{t-s}^b(\cdot,y),f_n^x\right\rangle_{\mu} - p_{t-s}^b(x,y)\right)^2d\mu(y)ds \\&\leq \liminf_{m\to\infty}\int_0^t\sum_{k=1}^{\infty}e^{-2\lambda_k^b(t-s)}\left[\left\langle \varphi_k^b,f_n^x-f_m^x\right\rangle_{\mu}\right]^2ds\\
&=\liminf_{m\to\infty}\int_0^t\int_0^1\left(\int_0^1 p_{t-s}^b(y,z)(f_n^x(z)-f_m^x(z)d\mu(z))\right)^2d\mu(y)ds.\\
\end{align*}
and further
\begin{align}
&\int_0^t\int_0^1\left(\int_0^1 p_{t-s}^b(y,z)(f_n^x(z)-f_m^x(z)d\mu(z))\right)^2d\mu(y)ds\notag \\ &\leq
\frac{e^{2t}}{2} \int_0^t\int_0^1 \rho_1^b(x,y)(f^x_n(y)-f^x_m(y))(f^x_n(z)-f^x_m(z))d\mu(y)d\mu(z)\label{spatial_convergence_lemma_1}\\
&\leq 4e^{2t}L_1(r_{\max}^n+r_{\max}^m)\label{spatial_convergence_lemma_2},
\end{align}
where we have used Lemma \ref{estimate_resolvent_lemma} in \eqref{spatial_convergence_lemma_1} and \eqref{resolvent_approximation_estimate} in 
\eqref{spatial_convergence_lemma_2}. We conclude
\begin{align*}
\int_0^t\int_0^1\left( \left\langle p_{t-s}^b(\cdot,y),f_n^x\right\rangle_{\mu} - p_{t-s}^b(x,y)\right)^2d\mu(y)ds &\leq \liminf_{m\to\infty} 4L_1e^{2t}\left(r_{\max}^n+r_{\max}^m\right)\\
&\leq 4L_1e^{2t}r_{\max}^n.
\end{align*}
\end{proof}

We are now able to prove stochastic continuity properties of $v_1$ and $v_2$ which are defined as follows for $(t,x)\in[0,T]\times[0,1]$ and $v_0\in\mathcal{S}_{q,T}$
\begin{align}
v_1(t,x)&\coloneqq \int_0^t\int_0^1p_{t-s}^b g(s,v_0(s,y))\xi(s,y)d\mu(y)ds,\label{heat_1}\\
v_2(t,x)&\coloneqq \int_0^t\int_0^1p_{t-s}^b f(s,v_0(s,y))d\mu(y)ds.\label{heat_2}
\end{align}
\begin{satz} \label{stochastic_continuity_heat}
Let $q\geq 2$ be fixed. Then, there exists a constant $C_8>0$ such that for all $v_0\in \mathcal{S}_{q,T}$ $v_1$ and $v_2$ are well-defined and it holds for all $s,t\in[0,T], x,y\in[0,1]$
\begin{align*}
\mathbb{E}\left(|v_1(t,x)-v_1(t,y)|^q\right)&\leq C_8\left(1+\lVert v_0\rVert_{q,T}^q\right)|x-y|^{\frac{q}{2}},\\
\mathbb{E}\left(|v_1(s,x)-v_1(t,x)|^q\right)&\leq C_8\left(1+\lVert v_0\rVert_{q,T}^q\right)|s-t|^{q\left(\frac{1}{2}-\frac{\gamma\delta}{2}\right)},\\
\mathbb{E}\left(|v_2(t,x)-v_2(t,y)|^q\right)&\leq C_8\left(1+\lVert v_0\rVert_{q,T}^q\right)|x-y|^{\frac{q}{2}},\\
\mathbb{E}\left(|v_2(s,x)-v_2(t,x)|^q\right)&\leq C_8\left(1+\lVert v_0\rVert_{q,T}^q\right)|s-t|^{q\left(\frac{1}{2}-\frac{\gamma\delta}{2}\right)}.
\end{align*}
\end{satz}
\begin{bem}
Note that $\gamma\delta<1$ is required in this Proposition. But this is no restriction since it is equivalent to \[\max_{i=1,...,N}\frac{\log\mu_i}{\log(\mu_ir_i)}<1\] and this is fulfilled since it holds for $0<\mu_i,r_i<1$
\begin{align*}
\log(\mu_ir_i) < \log(\mu_i)<0.
\end{align*} 
\end{bem}
\begin{proof}
First, we consider $v_1$. Since $(t,y)\to p^b_{t}(x,y)$ is continuous on $(0,T]\times[0,1]$ for $x\in[0,1]$ and $g$ and $v_0$ are predictable, the integrand is predictable. In order to prove the spatial estimate for $v_1$, let $t\in[0,T],x,y\in[0,1]$ be fixed. Then, there exists a constant $C_q$ such that
\begin{align}
\mathbb{E}&\left(|v_1(t,x)-v_1(t,y)|^q\right)\notag\\&=
\mathbb{E}\left(\left|\int_0^t\int_0^1\left(p_{t-s}^b(x,z)-p_{t-s}^b(y,z)\right)g(s,v_0(s,z))\xi(s,z)d\mu(z)ds \right|^q\right)\notag\\
&\leq C_q\left(\mathbb{E}\left(\left|\int_0^t\int_0^1\left(p_{t-s}^b(x,z)-p_{t-s}^b(y,z)\right)^2g(s,v_0(s,z))^2d\mu(z)ds \right|^{\frac{q}{2}}\right)\right)^{\frac{2}{q}\frac{q}{2}}\label{v1_heat_spatial1}\\
&\leq C_q\left|\int_0^t\int_0^1\left|\left(p_{t-s}^b(x,z)-p_{t-s}^b(y,z)\right)^q\mathbb{E}\left(g(s,v_0(s,z))^q\right)\right|^{\frac{2}{q}}d\mu(z)ds \right|^{\frac{q}{2}}\label{v1_heat_spatial2}\\
&= C_q\left|\int_0^t\int_0^1\left(p_{t-s}^b(x,z)-p_{t-s}^b(y,z)\right)^2\left|\mathbb{E}\left(g(s,v_0(s,z))^q\right)\right|^{\frac{2}{q}}d\mu(z)ds \right|^{\frac{q}{2}}\notag\\
&\leq  2^{q-1}C_q\left(\lVert M\rVert_{q,T}^q+L^q\lVert v_0\rVert_{q,T}^q\right)\left|\int_0^t\int_0^1\left(p_{t-s}^b(x,z)-p_{t-s}^b(y,z)\right)^2d\mu(z)ds \right|^{\frac{q}{2}},\label{v1_heat_spatial3}
\end{align}
where we have used the Burkholder-Davis-Gundy inequality (see e.g. \cite[Theorem B.1]{KAS}) \eqref{v1_heat_spatial1}, Minkowski's integral inequality in \eqref{v1_heat_spatial2} and the relation 
\begin{align}
\mathbb{E}\left(\left|g(s,v_0(s,y))\right|^q\right)\leq \mathbb{E}\left( |M(s)|+L|u(s,y)|\right)^q\leq 2^{q-1}\left(\mathbb{E}\left(\left|M(s)\right|^q\right)+L^q\mathbb{E}\left(\left|v_0(s,y)\right|^q\right)\right).\label{g_heat_estimate}
\end{align}
in \eqref{v1_heat_spatial3}.
We proceed by estimating the integral term in \eqref{v1_heat_spatial3}, whereby we first treat the case $x,y\in F$. By Lemma \ref{estimate_resolvent_lemma_2}, Lemma \ref{estimate_resolvent_lemma}, and Lemma \ref{resolvent_density_lipschitz},
\begin{align*}
&\left|\int_0^t\int_0^1\left(p_{t-s}^b(x,z)-p_{t-s}^b(y,z)\right)^2d\mu(z)ds\right| \\
&=
\lim_{n\to\infty} \left|\int_0^t\int_0^1\left(\left\langle p_{t-s}^b(\cdot,z),f_n^x-f_n^y\right\rangle_{\mu}\right)^2d\mu(z)ds\right| \\
&\leq \frac{e^{2t}}{2}\lim_{n\to\infty} \left|\int_0^t\int_0^1\rho_1(z_1,z_2)(f_n^x(z_1)-f_n^y(z_1))(f_n^x(z_2)-f_n^y(z_2)) d\mu(z_1)d\mu(z_2)\right|\\
&=\frac{e^{2t}}{2}\left|\rho_1(x,x)-2\rho_1(x,y)+\rho_1(y,y)\right|\\
&\leq e^{2t}L_1|x-y|.
\end{align*}
Now, let $x,y\in F^c$ such that there exists an $i\in\mathbb{N}$ with $(x,y)\in(a_i,b_i)$ and we assume $x<y$. Then, since $b_i,a_i\in F$, the previous calculation implies
\begin{align*}
&\left|\int_0^t\int_0^1\left(p_{t-s}^b(x,z)-p_{t-s}^b(y,z)\right)^2d\mu(z)ds\right|\\
&\leq \int_0^t\sum_{k=1}^{\infty}e^{-2\lambda_k(t-s)}\left(\varphi_k(x)-\varphi_k(y)\right)^2ds\\
&\leq \int_0^t \left(\frac{x-y}{b_i-a_i}\right)^2\sum_{k=1}^{\infty}e^{-2\lambda_k(t-s)}\left(\varphi_k(b_i)-\varphi_k(a_i)\right)^2 ds\\
&\leq TL_1e^{2t}\left(\frac{x-y}{b_i-a_i}\right)^2|b_i-a_i|\\
&\leq TL_1e^{2t}\frac{(x-y)^2}{b_i-a_i}\\
&\leq TL_1e^{2t}|x-y|.
\end{align*}
The remaining cases for $x,y\in[0,1]$ follow by using the triangle inequality for the norm \linebreak $L^2([0,T]\times~\![0,1],\lambda^1\times~\!\mu)$. Consequently, for all $(x,y)\in[0,1]$
\begin{align*}
&\left|\int_0^t\int_0^1\left(p_{t-s}^b(x,z)-p_{t-s}^b(y,z)\right)^2d\mu(z)ds\right|^{\frac{1}{2}}\leq 3^{\frac{1}{2}} T^{\frac{1}{2}}e^{-\frac{1}{2}}L_1^{\frac{1}{2}}e^{t}|x-y|^{\frac{1}{2}}.
\end{align*}
We conclude
\begin{align*}
\mathbb{E}&\left(|v_1(t,x)-v_1(t,y)|^q\right)\leq
 3^{\frac{q}{2}}2^{q-1}e^{\frac{t}{q}}C_qL_1^{\frac{q}{2}}\left(\lVert M\rVert_{q,T}^q+L^q\lVert v_0\rVert_{q,T}^q\right)|x-y|^{\frac{q}{2}}.
\end{align*}

This proves the spacial estimate since the last integral is finite. We now turn to the temporal esimate. Let $t,s\in[0,T]$ with $s<t$ and $x\in[0,1]$. Then, by using the Burkholder-Davis-Gundy inequality, Minkowski's integral inequality and inequality \eqref{g_heat_estimate}, we get
\begin{align}
&\mathbb{E}\left(|v_1(t,x)-v_1(s,x)|^q\right)\notag\\
&\leq C_q\left|\int_0^t\int_0^1\left|\left(p_{t-u}^b(x,y)-p^b_{s-u}(x,y)\mathds{1}_{[0,s]}(u)\right)^2\mathbb{E}\left(g(s,v_0(s,y))^q\right)\right|^{\frac{2}{q}}d\mu(y)du \right|^{\frac{q}{2}}\notag\\
&\leq  2^{q-1}C_q\left(\lVert M\rVert_{q,T}^q+L^q\lVert v_0\rVert_{q,T}^q\right)\left|\int_0^t\int_0^1\left(p_{t-u}^b(x,y)-p^b_{s-u}(x,y)\mathds{1}_{[0,s]}(u)\right)^2d\mu(y)du \right|^{\frac{q}{2}}.\notag
\end{align}
We split the above integral in the time intervals $[0,s]$ and $(s,t]$ and get for the first part
\begin{align}
\int_0^s\int_0^1\left(p_{t-u}^b(x,y)-p^b_{s-u}(x,y)\right)^2d\mu(y)du
&=\int_0^s\int_0^1\left(p_{u}^b(x,y)-p^b_{u+t-s}(x,y)\right)^2d\mu(y)du\notag \\
\begin{split}
&=\int_0^s\int_0^1p_{2u}^b(x,x)-2p^b_{2u+t-s}(x,x)\\&~~~~~~~~~~~~+p_{2(u+t-s)}^b(x,x)d\mu(y)du\notag 
\end{split}
\\
&=2^{-\gamma}C_7(2T)\int_0^s u^{-\gamma\delta}-(u+t-s)^{-\gamma\delta}du\label{v_1_temporal_heat_1} \\
&=\frac{2^{-\gamma}}{1-\gamma\delta} C_7(2T) \left(s^{1-\gamma\delta}-t^{1-\gamma\delta}+(t-s)^{1-\gamma\delta}\right)\notag \\
&\leq\frac{2^{-\gamma\delta}}{1-\gamma\delta} C_7(2T) (t-s)^{1-\gamma\delta},\notag 
\end{align}
where we have used Proposition \ref{heat_kernel_estimates}\ref{heat_kernel_estimate_3} in \eqref{v_1_temporal_heat_1}.  For the second part by using Proposition  \ref{heat_kernel_estimates}\ref{heat_kernel_estimate_1}
\begin{align*}
\int_s^t\int_0^1p_{t-u}^b(x,y)^2d\mu(y)du\notag 
&= \int_0^{t-s}p_{2s}^b(x,x)du\notag \\
&\leq 2^{-\gamma}C_5(2T)\int_0^{t-s} u^{-\gamma\delta}du\\
&\leq \frac{2^{-\frac{\gamma}{4}}}{1-\gamma} C_5(2T)(t-s)^{1-\gamma\delta}.
\end{align*}
For the estimates for $v_2$ use Jensen's inequality instead of the Burkholder-Davis-Gundy inequality and the rest of the proof works similarly. 
\end{proof}
\begin{kor}\label{Sqt_korollar_heat}
Let $q\geq 2$ and $v_0\in\mathcal{S}_{q,T}$. Then, $v_1$ and $v_2$ defined as in \eqref{heat_1}-\eqref{heat_2} are elements of $\mathcal{S}_{q,T}$.
\end{kor}
\begin{proof}
By setting $s=0$ in Proposition \ref{stochastic_continuity_heat} we obtain $\left\lVert v_i\right\rVert_{q,T}<\infty$, $i=1,2$. We need to show that $v_1$ is predictable. 
For $n\in\mathbb{N}$ let
\begin{align*}
v_1^n(t,x)=\sum_{i,j=0}^{2^n-1}v_1\left(\frac{i}{2^n}T,\frac{j}{2^n}\right)\mathds{1}_{\left(\frac{i}{2^n}T,\frac{i+1}{2^n}T\right]}(t)\mathds{1}_{\left(\frac{j}{2^n},\frac{j+1}{2^n}\right]}(x), ~~(t,x)\in[0,T]\times[0,1].
\end{align*}
It holds evidently $\left\lVert v_1^n\right\rVert_{q,T}<\infty$.  To prove that $v_1^n$ is predictable, we show that $v_1^n$ is the $\mathcal{S}_{q,T}$-limit of a sequence of simple functions. To this end, let for $N\geq 1$
\begin{align*}
v_1^{n,N}(t,x)=v_1^n(t,x)\wedge N, ~~ t\in[0,T],~ x\in[0,1].
\end{align*}
This defines a simple function since $v_1\left(\frac{i}{2^n}T,\frac{j}{2^n}\right)\wedge N$ is $\mathcal{F}_{\frac{iT}{2^n}}$-measurable and bounded. It converges in $\mathcal{S}_{q,T}$ to $v_1^n$, which can be seen as follows:
\begin{align*}
&\lim_{N\to\infty}\sup_{t\in[0,T]}\sup_{x\in[0,1]}\left\lVert v_1^n(t,x)-v_1^{n,N}(t,x)\right\rVert_{L^q(\Omega)}\\
&\leq  \lim_{N\to\infty}\sup_{t\in[0,T]}\sup_{x\in[0,1]}\sum_{i,j=0}^{2^n-1}\left\lVert v_1\left(\frac{i}{2^n}T,\frac{j}{2^n}\right)- v_1\left(\frac{i}{2^n}T,\frac{j}{2^n}\right)\wedge N\right\rVert_{L^q(\Omega)}\\
&= \lim_{N\to\infty}\sum_{i,j=0}^{2^n-1}\left\lVert v_1\left(\frac{i}{2^n}T,\frac{j}{2^n}\right)- v_1\left(\frac{i}{2^n}T,\frac{j}{2^n}\right)\wedge N\right\rVert_{L^q(\Omega)}
\\ &=0,
\end{align*}
where the last equation follows from the monotone convergence theorem. We conclude that $v_1^n$ is  predictable for $n\in\mathbb{N}$ .
By Proposition  \ref{stochastic_continuity_heat},
\begin{align*}
\left\lVert v_1-v_1^n\right\rVert_{q,T}&\leq \sup_{|s-t|<\frac{T}{n}}\sup_{|x-y|<\frac{1}{n}}\left\lVert v_1(s,x)-v_1(t,y)\right\rVert_{L^q(\Omega)}\\
&\leq  \sup_{|s-t|<\frac{T}{n}}\sup_{|x-y|<\frac{1}{n}}\left\lVert v_1(s,x)-v_1(t,x)\right\rVert_{L^q(\Omega)}\\
&~~~+\sup_{|s-t|<\frac{T}{n}}\sup_{|x-y|<\frac{1}{n}}\left\lVert v_1(t,x)-v_1(t,y)\right\rVert_{L^q(\Omega)}\\
&\leq C_6\left(\left(\frac{T}{n}\right)^{\frac{1}{2}-\frac{\gamma\delta}{2}}+\left(\frac{1}{n}\right)^{\frac{1}{2}}\right)\to 0, ~n\to\infty.
\end{align*}
Hence, $v_1$ is predictable. The predictability of $v_2$ follows analogously.
\end{proof}
After these preparations we can follow the methods of \cite[Theorem 6.9]{HYC} in order to establish existence and uniqueness.
\begin{thm}
Assume Condition \ref{hypo_heat} with $q\geq 2$.  Then, SPDE \eqref{spde_heat} has a unique mild solution in $\mathcal{S}_{q,T}$.
\end{thm}
\begin{proof}
\textit{Uniqueness:} Let $u,\widetilde u\in\mathcal{S}_{q,T}$ be mild solutions to \eqref{spde_heat}. Then $v\coloneqq u-\widetilde u\in\mathcal{S}_{2,T}$.
With $G(t)\coloneqq \sup_{x\in[0,1]}\mathbb{E}\left[v^2(t,x)\right]$
we calculate for $(t,x)\in[0,T]\times[0,1]$
\begin{align}
\mathbb{E}\left[v(t,x)^2\right]
&\leq 2T\mathbb{E}\left[\int_0^t\int_0^1\left(p_{t-s}^b(x,y)\right)^2\left(f(s,u(s,y))-f\left(s,\widetilde u(s,y)\right)\right)^2d\mu(y)ds\right]\notag\\
&~~+2\mathbb{E}\left[\int_0^t\int_0^1\left(p_{t-s}^b(x,y)\right)^2\left(g(s,u(s,y))-g\left(s,\widetilde u(s,y)\right)\right)^2d\mu(y)ds\right]\label{uniqueness_heat_0}\\
&\leq 2(T+1)L^2\mathbb{E}\left[\int_0^t\int_0^1 v^2(s,y)\left(p_{t-s}^b(x,y)\right)^2d\mu(y)ds\right]\notag\\
&\leq 2(T+1)L^2\int_0^tG(s)\int_0^1\left(p_{t-s}^b(x,y)\right)^2d\mu(y)ds\notag\\
&=2(T+1)L^2\int_0^tG(s)p_{2(t-s)}^b(x,x)ds\notag\\
&\leq 2^{1-\gamma\delta}(T+1)C_5(2T) L^2\int_0^tG(s)(t-s)^{-\gamma\delta} ds\label{uniqueness_heat_1},
\end{align}
where we have used Walsh's isometry and Hölder's inequality
in \eqref{uniqueness_heat_0} and Proposition \ref{heat_kernel_estimates}\ref{heat_kernel_estimate_1} in \eqref{uniqueness_heat_1}. It follows 
\begin{align*}
G(t)&\leq 2^{1-\gamma\delta}(T+1)C_5(2T) L^2\int_0^tG(s)(t-s)^{-\gamma\delta} ds
\end{align*}
and by setting $h_n=G$ in \cite[Lemma 3.3]{WI} we obtain that $G(t)=0$ for $t\in[0,T]$.
We conclude $u(t,x)=\widetilde u(t,x)$ almost surely for every $(t,x)\in[0,T]\times[0,1]$.\par
\textit{Existence:} As usual, we use Picard iteration to find a solution. For that, let $u_1=0\in\mathcal{S}_{q,T}$ and for $n\geq 1$
\begin{align}
\begin{split}
u_{n+1}(t,x)=\int_0^1&p_t^b(x,y)u_0(y)d\mu(y)+\int_0^t\int_0^1p_{t-s}^b(x,y)f(s,u_n(s,y))d\mu(y)ds\\&+\int_0^t\int_0^1p_{t-s}^b(x,y)g(s,u_n(s,y))\xi(s,y)d\mu(y)ds. \label{picard_equ_heat}
\end{split}
\end{align}
Let $n\geq 1$, assume that $u_n\in\mathcal{S}_{q,T}$ and define $u_{n+1}$ as in \eqref{picard_equ_heat}. The last two terms on the right-hand side are elements of $\mathcal{S}_{q,T}$ by Proposition \ref{Sqt_korollar_heat}. The first term is predictable because it is $\mathcal{F}_0$-measurable and thus adapted and almost surely continuous due to the dominated convergence theorem and Proposition \ref{heat_kernel_estimates}\ref{heat_kernel_estimate_1}. Furthermore, by Minkowski's integral inequality
\begin{align*}
\mathbb{E}\left[\left|\int_0^1p_t^b(x,y)u_0(y)d\mu(y)\right|^q\right]&\leq\left(\int_0^1p_t^b(x,y)\mathbb{E}\left[|u_0(y)|^q\right]^{\frac{1}{q}}d\mu(y)\right)^q\\
&\leq \left\lVert u_0\right\rVert_q^q\left|\int_0^1p_t^b(x,y)d\mu(y)\right|^q\\
&\leq \left\lVert u_0\right\rVert_q^q.
\end{align*}
In the last inequality we have used the Markov property of $T_t$ and the continuity of $\int_0^1p_t^b(\cdot,y)d\mu(y)$ by dominated convergence and Proposition \ref{heat_kernel_estimates}\ref{heat_kernel_estimate_1} to get with $h=1$
\begin{align}
\int_0^1p_t^b(x,y)d\mu(y)=\left|T_t^bh(x)\right|\leq 1, ~~ x\in[0,1] \label{leq1}
\end{align}
for $t\in[0,T]$. It follows that $u_{n+1}\in \mathcal{S}_{q,T}$.\par
 We prove that $(u_n)_{n\in\mathbb{N}}$ is a Cauchy sequence in $\mathcal{S}_{q,T}$. Let $v_n=u_{n+1}-u_n\in\mathcal{S}_{q,T}$. By using Hölder's and the Burkholder-Davis-Gundy inequality, the Lipschitz property of $f$ and $g$ as well as Minkowski's integral inequality we get 
\begin{align*}
&\mathbb{E}\left[v_{n+1}(t,x)^q\right]\\
&\leq 2^{q-1}T^{\frac{q}{2}} \mathbb{E}\left[\left|\int_0^t\int_0^1p_{t-s}^b(x,y)^2\left(f(s,u_{n+1}(s,y))-f\left(s,u_n(s,y)\right)\right)^2d\mu(y)ds\right|^{\frac{q}{2}}\right]\\
&~~+ 2^{q-1}C_q\mathbb{E}\left[\left|\int_0^t\int_0^1p_{t-s}^b(x,y)^2\left(g(s,u_{n+1}(s,y))-g\left(s,u_n(s,y)\right)\right)^2d\mu(y)ds\right|^{\frac{q}{2}}\right]\\
&\leq 2^{q-1}\left(T^{\frac{q}{2}}+C_q\right) L^q\mathbb{E}\left[\left|\int_0^t\int_0^1 p_{t-s}^b(x,y)^2v_n^2(s,y)d\mu(y) ds\right|^{\frac{q}{2}}\right]\\
&\leq  2^{q-1}\left(T^{\frac{q}{2}}+C_q\right)\left(\int_0^t\int_0^1 p_{t-s}^b(x,y)^2\left(\mathbb{E}\left[|v_n(s,y)|^q\right]\right)^{\frac{2}{q}}	d\mu(y)ds\right)^{\frac{q}{2}}.
\end{align*}
Set  $H_n(t)=\sup_{x\in[0,1]}\left(\mathbb{E}\left[|v_n(t,y)|^q\right]\right)^{\frac{2}{q}}$ for $n\geq 1,~ t\in[0,T]$. Then for every $n\geq 2$ there exists a constant $c_n$ such that $|H_n(t)|\leq c_n$ for every $t\in[0,T]$. With Proposition \ref{heat_kernel_estimates}\ref{heat_kernel_estimate_1} it follows that there exists $C>0$ such that for $(t,x)\in[0,T]\times[0,1]$ and $n\in\mathbb{N}$ 
\begin{align*}
&\left(\mathbb{E}\left[v_{n+1}(t,x)^q\right]\right)^{\frac{2}{q}}\leq C\int_0^t H_n(s)p_{2(t-s)}^b(x,x)d\mu(y)ds\\
&\leq 2^{-\gamma\delta}C_5(2T)C\int_0^t H_n(s)(t-s)^{-\gamma\delta}ds
\end{align*}
and thus for $t\in[0,T]$ and $n\in\mathbb{N}$
\begin{align*}
H_{n+1}(t)\leq  2^{-\gamma\delta}C_5(2T)C\int_0^t H_n(s)(t-s)^{-\gamma\delta}ds.
\end{align*}
From \cite[Lemma 3.3]{WI} it follows that there exists a constant $C'$ and a $k\in\mathbb{N}$ such that for $n,m\geq 1$, $t\in[0,T]$
\begin{align*}
H_{n+mk}(t)\leq\frac{C'^m}{(m-1)!}\int_0^t H_n(s)(t-s)ds.
\end{align*}
Therefore $\sum_{m\geq 1}\sqrt{H_{n+mk}}$ converges uniformly in $[0,T]$, which can be verified by the ratio test using that  $\sqrt{\frac{H_{n+(m+1)k}(t)}{H_{n+mk}(t)}}\leq \sqrt{\frac{C'}{m}}$ for $n\geq 1$.
 We conclude  \[\sup_{t\in[0,T]}\sqrt{H_n(t)}\to 0,~ n\to\infty,\] 
which implies the same for $\left\lVert v_n\right\rVert_{q,T}$. Hence, $(u_n)_{n\geq 1}$ is Cauchy in $\mathcal{S}_{q,T}$ with limit denoted by $u$. To verify that $u$ satisfies \eqref{mild_solution_heat} we take the limit in $L^q(\Omega)$ for $n\to\infty$ on both sides of \eqref{picard_equ_heat} for every $(t,x)\in[0,T]\times[0,1]$. We obtain $u(t,x)$ on the left-hand side for any $(t,x)\in[0,T]\times[0,1]$. For the right-hand side we note that there exists $C''>0$ such that for $(t,x)\in[0,T]\times[0,1]$  
\begin{align*}
&\mathbb{E}\left[\left|\int_0^t\int_0^1p_{t-s}^b(x,y)\left(f(s,u(s,y))-f\left(s,u_n(s,y)\right)\right)\xi(s,y)d\mu(y)ds\right|^{q}\right]\\
~~&+\mathbb{E}\left[\left|\int_0^t\int_0^1p_{t-s}^b(x,y)\left(g(s,u(s,y))-g\left(s,u_n(s,y)\right)\right)d\mu(y)ds\right|^{q}\right]\\
&\leq C''\left(\int_0^t\int_0^1 p_{t-s}^b(x,y)^2\left(\mathbb{E}\left[|u(s,y)-u_n(s,y)|^q\right]\right)^{\frac{2}{q}}	d\mu(y)ds\right)^{\frac{q}{2}},
\end{align*}
which goes to zero as $n$ tends to infinity with the same argumentation as before.
\end{proof}

Preliminary for the formulation of the main result of this section, we define $u^{\sto}\in \mathcal{S}_{q,T}$ by 
\begin{align*}
u^{\sto}(t,x)=&\int_0^t\int_0^1p_{t-s}^b(x,y)f(s,u(s,y))d\mu(y)ds
+\int_0^t\int_0^1p_{t-s}^b(x,y)g(s,u(s,y))\xi(s,y)d\mu(y)ds
\end{align*}
almost surely for $(t,x)\in[0,T]\times[0,1]$.
We consider this process because the regularity of a version of $u$ is, in general, restricted by the regularity of  $u-u^{\sto}$. 
Furthermore, we introduce the normed product space $([0,T]\times[0,1],||\cdot||_{2})$, where $\lVert\cdot\rVert_{2}$ is the Euclidian norm.
\begin{thm}\label{main_result_heat}
Assume Condition \ref{hypo_heat} with $q\geq 2$. Then, there exists a version of $u^{\sto}$, denoted by $\tilde  u^{\sto}$, such that the following holds:
\begin{enumerate}[label=(\roman*)]
\item \label{main_result_heat_2} If $q>2$ and $t\in[0,T]$,  $\tilde u^{\sto}(t,\cdot)$ is a.s. essentially $\frac{1}{2}-\frac{1}{q}$-Hölder continuous on $[0,1]$.
\item \label{main_result_heat_3} If $q>\left(\frac{1}{2}-\frac{\gamma\delta}{2}\right)^{-1}$ and $x\in[0,1]$, $\tilde u^{\sto}(\cdot,x)$ is a.s. essentially $\frac{1}{2}-\frac{\gamma\delta}{2}-\frac{1}{q}$-Hölder continuous on $[0,T]$.
\item \label{main_result_heat_1}  If $q>4\vee \left(2\left(\frac{1}{2}-\frac{\gamma\delta}{2}\right)^{-1}\right)$, $\tilde u^{\sto}$ is a.s. essentially $\left(\frac{1}{2}-\frac{\gamma\delta}{2}-\frac{2}{q}\right)\wedge \left(\frac{1}{2}-\frac{2}{q}\right)$-Hölder continuous on $[0,T]\times[0,1]$.
\end{enumerate}
\end{thm}
\begin{proof}
The continuity properties in part \ref{main_result_heat_2} and \ref{main_result_heat_3} of a version of $u^{\sto}$ follow immediately from Proposition \ref{stochastic_continuity_heat} and Kolmogorov's continuity theorem. 
Further, for $(t,x)\in[0,T]\times[0,1]$ by using Proposition \ref{stochastic_continuity_heat} with $v_0=u\in\mathcal{S}_{q,T}$,
\begin{align*}
\mathbb{E}\left[\left(\tilde u^{\sto}(s,x)-\tilde u^{\sto}(t,y)\right)^q\right]&\leq 
2^{q-1}\left(\mathbb{E}\left[\left(\tilde u^{\sto}(s,x)-\tilde u^{\sto}(t,x)\right)^q\right]
+\mathbb{E}\left[\left(\tilde u^{\sto}(t,x)-\tilde u^{\sto}(t,y)\right)^q\right]\right)
\\
&\leq  2^qC_8\left(1+\lVert u_0\rVert_{q,T}^q\right)\left(|x-y|^{\frac{q}{2}}+|s-t|^{q\left(\frac{1}{2}-\frac{\gamma\delta}{2}\right)}\right)\\
&\leq  2^qC_8\left(1+\lVert u_0\rVert_{q,T}^q\right)\left(\left(|x-y|^2\right)^{\frac{q}{4}}+\left(|s-t|^2\right)^{\frac{q}{2}\left(\frac{1}{2}-\frac{\gamma\delta}{2}\right)}\right)\\
&\leq 2^qC_8 T^{q\left(\frac{1}{2}-\frac{\gamma\delta}{2}\right)}\left(1+\lVert u_0\rVert_{q,T}^q\right)\Bigg(\left(|x-y|^2\right)^{\frac{q}{4}}\\&~~~~~~~~~~~~~~~~~~~~~~~~~~~~~~~~~~~~~~~~~~~~+\left(\left|\frac{s}{T}-\frac{t}{T}\right|^2\right)^{\frac{q}{2}\left(\frac{1}{2}-\frac{\gamma\delta}{2}\right)}\Bigg)\\
&\leq 2^qC_8 T^{q\left(\frac{1}{2}-\frac{\gamma\delta}{2}\right)}\left(1+\lVert u_0\rVert_{q,T}^q\right) \left\lVert \left(x-y,s-t\right)\right\rVert_{2}^{\frac{q}{2}\wedge \left(q\left(\frac{1}{2}-\frac{\gamma\delta}{2}\right)\right)}.
\end{align*}
In the last inequality we have used that $|x-y|\leq 1$ and $\left|\frac{s}{T}-\frac{t}{T}\right|\leq 1$ as well as $\frac{q}{4}\geq 1$ and $\frac{q}{2}\left(\frac{1}{2}-\frac{\gamma\delta}{2}\right)\geq 1$ due to the assumption of part \ref{main_result_heat_1}. The result follows from Kolmogorov's continuity theorem in two dimensions (see, e.g., \cite[Remark 21.7]{KW}).

\end{proof}

\begin{figure}[t]
\centering
\includegraphics[scale=0.6]{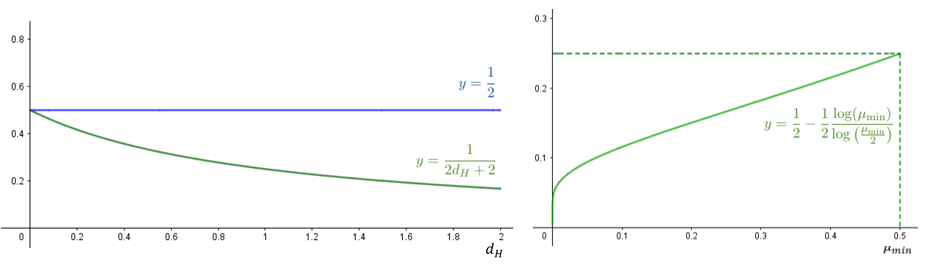}
\caption{Hölder exponent graphs (see Example \ref{hoelder_exa})}
\label{hoelder_plot}
\end{figure} 

\begin{exa}\label{hoelder_exa}
\begin{enumerate}[label=(\roman*)]
\item If $u_0, f$ and $g$ satisfy Assumption \ref{hypo_heat} and are uniformly bounded, $q$ can be chosen arbitrarily large such that we obtain $\frac{1}{2}$ as ess. spatial and $\frac{1}{2}-\frac{\gamma\delta}{2}$ as ess. temporal Hölder exponent. If, in addition, the measure $\mu$ is chosen as the Hausdorff measure on a given Cantor-like set with Hausdorff dimension $d_H$, then
\begin{align*}
\frac{1}{2}-\frac{\gamma\delta}{2}=\frac{1}{2}-\frac{1}{2}\max_{1\leq i\leq n}\frac{\log \left(r_i^{d_H}\right)}{\log \left(r_i^{d_H+1}\right)}=\frac{1}{2}-\frac{d_H}{2d_H+2}=\frac{1}{2d_H+2}.
\end{align*}
Under these conditions, wet get the same terms in the p.c.f. fractal case with $1\leq d_H <2$ (compare \cite[Theorem 6.14]{HYC}). These terms are visualized on the left-hand side of Figure \ref{hoelder_plot} for $0<d_H<2$. 
\item If $\mu$ is not the Hausdorff measure on a given Cantor-like set, then
\begin{align*}
\frac{1}{2}-\frac{\gamma\delta}{2}<\frac{1}{2d_H+2}.
\end{align*}
As an example, consider the weighted IFS given by $S_1(x)=0.5x,~ S_2(x)=0.5x+0.5$ and weights $\mu_1,\mu_2\in(0,1)$. It follows
\begin{align*}
\frac{1}{2}-\frac{\gamma\delta}{2}= \frac{1}{2}-\frac{1}{2}\max_{i=1,2}\frac{\log(\mu_i)}{\log(\mu_ir_i)}=\frac{1}{2}-\frac{1}{2}\frac{\log(\mu_{\min})}{\log\left(\frac{\mu_{\min}}{2}\right)},
\end{align*} 
which goes to zero as $\mu_{\min}$ tends to zero. This is visualized on the right-hand side of  Figure \ref{hoelder_plot}.
\end{enumerate}
\end{exa}
\subsection{Intermittency}
According to \cite{KA} we call the mild solution $u$ weakly intermittent on $[0,1]$ if the lower and the upper moment Lyapunov exponents which are respectively the functions $\gamma$ and $\bar \gamma$ defined by
\begin{align*}
\gamma(p,x)\coloneqq \liminf_{t\to\infty}\frac{1}{t}\log\mathbb{E}\left[u(t,x)^p\right], ~~ 
\bar\gamma(p,x)\coloneqq \limsup_{t\to\infty}\frac{1}{t}\log\mathbb{E}\left[u(t,x)^p\right], ~~ p\in(0,\infty), x\in[0,1]
\end{align*}
satisfy
\begin{align*}
\gamma(2,x)>0,~~~ \bar\gamma(p,x)<\infty, ~~ ~p\in [2,\infty), x\in[0,1].
\end{align*}

In this section we make the following additional assumption:
\begin{hyp}
We assume that Condition \ref{hypo_heat} holds for some $q\geq 2$. Moreover, we assume that $u_0$ is non-negative and that $f$ and $g$ satisfy the following Lipschitz and linear growth condition: For all $(w,t,x)\in\Omega\times[0,\infty)\times\mathbb{R}$ there exists a constant $L>0$ such that
\begin{align*}
|f(\omega,t,x)-f(\omega,t,y)|+|g(\omega,t,x)-g(\omega,t,y)|&\leq L|x-y|,\\
|f(\omega,t,y)|+|g(\omega,t,y)|&\leq L(1+|x|).
\end{align*}
\end{hyp}

\begin{thm}\label{intermittency_upper_estimate_heat}
There exists $C_{10}>0$ such that for all $p\in[1,q]$, $(t,x)\in[0,\infty)\times[0,1]$
\begin{align*}
\left(\mathbb{E}\left[\left|u(t,x)\right|^p\right]\right)^{\frac{1}{p}}\leq \left(2\left\lVert u_0\right\rVert_q+1\right)e^{C_{10}p^{\frac{1}{1-\gamma}}t}.
\end{align*}
\end{thm}
\begin{bem}
If the weights are chosen as $\mu_i=r_i^{d_H}, i=1,..,N$, it holds $\gamma=\frac{d_H}{d_H+1}$ and therefore $\frac{1}{1-\gamma}=d_H+1$. Then, the above inequality reads as 
\[
\left(\mathbb{E}\left[\left|u(t,x)\right|^p\right]\right)^{\frac{1}{p}}\leq \left(2\left\lVert u_0\right\rVert_q+1\right)e^{C_{10}p^{d_H+1}t},
\] which holds in the same way for p.c.f. self-similar sets (compare \cite[Theorem 7.5]{HYC}).
\end{bem}
Preparing the proof of Proposition \ref{intermittency_upper_estimate_heat} we give an estimate on $p_t^b(x,x)$ for $x\in[0,1]$ which holds for all $t\geq 0$ instead of only on a finite time interval as in Proposition \ref{heat_kernel_estimates}\ref{heat_kernel_estimate_1}.
\begin{lem}\label{heat_kernel_estimate_intermittency}
Let $b\in\{N,D\}$. There exists $C_{11}>0$ such that for all $(t,x)\in[0,\infty)\times[0,1]$
\begin{align*}
p_t^b(x,x)\leq C_{11}\left(1+t^{-\gamma\delta}\right).
\end{align*}
\end{lem}
\begin{proof}
Let $T>0$. By \eqref{eigenwertabsch} and Theorem \eqref{eigenfunction_estimate} it holds for each $x\in[0,1]$
\begin{align*}
\left|p_t^b(x,x)\right|\leq \sum_{k\geq 1}e^{-C_0 k^{\frac{1}{\gamma}}t}C_2^2k^{\delta},
\end{align*}
which converges uniformly on $[T,\infty)$ (see \cite[Lemma 5.1.4]{KA}). Further, note that $\lambda_1^N=0, \varphi_1^N\equiv 1, \varphi_1^D>0$.
The dominated convergence theorem gives us
\begin{align*}
\lim_{t\to\infty} p_t^N(x,x)=1,~ \lim_{t\to\infty} p_t^D(x,x)=0
\end{align*}
uniformly for all $x\in[0,1]$. This along with Proposition \ref{heat_kernel_estimates}\ref{heat_kernel_estimate_3} implies the result.
\end{proof}
\begin{proof}[Proof of Theorem \ref{intermittency_upper_estimate_heat}]
We follow the methods of \cite[Theorem 7.5]{HYC}. Let $p\in[2,q], ~\alpha>0,~ (t,x)\in[0,\infty)\times[0,1]$. By using Minkowski's integral inequality and the Burkholder-Davis-Gundy inequality
\begin{align*}
e^{-\alpha t}\left(\mathbb{E}\left[\left|u(t,x)\right|^p\right]\right)^{\frac{1}{p}}
&\leq \left\lVert u_0\right\rVert_p+\left(\mathbb{E}\left[\left|\int_0^t\int_0^1 e^{-\alpha t}p_{t-s}^b(x,y)f(s,u(s,y))d\mu(y) ds\right|^p\right]\right)^{\frac{1}{p}}\\
&~~+2\sqrt{p}\left(\mathbb{E}\left[\left|\int_0^t\int_0^1 e^{-2\alpha t}p_{t-s}^b(x,y)^2g(s,u(s,y))^2 ds\right|^{\frac{p}{2}}\right]\right)^{\frac{1}{p}}.
\end{align*}
To estimate the first integral we use Minkowski's integral inequality and obtain
\begin{align*}
&\left(\mathbb{E}\left[\left|\int_0^t\int_0^1 e^{-\alpha t}p_{t-s}^b(x,y)f(s,u(s,y))d\mu(y) ds\right|^p\right]\right)^{\frac{1}{p}}\\
&~~ \leq \int_0^t\int_0^1 e^{-\alpha t}p_{t-s}^b(x,y)\left(\mathbb{E}\left[\left|f(s,u(s,y))\right|^p\right]\right)^{\frac{1}{p}} d\mu(y) ds\\
&~~ \leq L\int_0^t\int_0^1 e^{-\alpha t}p_{t-s}^b(x,y)\left(1+\left(\mathbb{E}\left[\left|f(s,u(s,y))\right|^p\right]\right)^{\frac{1}{p}}\right) d\mu(y) ds\\
&~~ \leq L\left(te^{-\alpha t} +\int_0^t\int_0^1 e^{-\alpha t}p_{t-s}^b(x,y)\sup_{z\in[0,1]}\left(\left(\mathbb{E}\left[\left|f(s,u(s,z))\right|^p\right]\right)^{\frac{1}{p}}\right) d\mu(y) ds\right)\\
&~~ \leq L\left(\frac{1}{\alpha}+\sup_{(s,z)\in[0,t]\times[0,1]}\left(\left(e^{-\alpha s} \mathbb{E}\left[\left|f(s,u(s,z))\right|^p\right]\right)^{\frac{1}{p}}\right)\int_0^t e^{-\alpha (t-s)} ds\right)\\
&~~\leq \frac{L}{\alpha}\left(1+\sup_{(s,z)\in[0,t]\times[0,1]}\left(\left(e^{-\alpha s} \mathbb{E}\left[\left|f(s,u(s,z))\right|^p\right]\right)^{\frac{1}{p}}\right)\right).
\end{align*}
In the second last inequality we have used that $t\to te^{-\alpha t}$ has its maximum at $t=\frac{1}{\alpha}$ which can be seen by differentiating. We turn to the second integral. By using Lemma \ref{heat_kernel_estimate_intermittency}

\begin{align*}
&\left(\mathbb{E}\left[\left|\int_0^t\int_0^1 e^{-2\alpha t}p_{t-s}^b(x,y)^2g(s,u(s,y))^2d\mu(y) ds\right|^{\frac{p}{2}}\right]\right)^{\frac{1}{p}}~~~~~~~~~~~~~~~~~~~\\
&~~ \leq \left(\int_0^t\int_0^1 e^{-2\alpha t}p_{t-s}^b(x,y)^2\left(\mathbb{E}\left[\left|g(s,u(s,y))\right|^{p}\right]\right)^{\frac{2}{p}} d\mu(y) ds\right)^{\frac{1}{2}}~~~~~~~~~~~~ ~~~~~~~
\end{align*}
\begin{align*}
&~~ \leq L\left(\int_0^t e^{-2\alpha t}p_{2(t-s)}^b(x,x)\sup_{z\in[0,1]}\left(1+\left(\mathbb{E}\left[\left|u(s,z)\right|^p\right]\right)^{\frac{1}{p}}\right)^2 ds\right)^{\frac{1}{2}}\\
&~~ = L\left(\int_0^t e^{-2\alpha (t-s)}p_{2(t-s)}^b(x,x)\sup_{z\in[0,1]}\left(e^{-\alpha s}+e^{-\alpha s}\left(\mathbb{E}\left[\left|u(s,z)\right|^p\right]\right)^{\frac{1}{p}}\right)^2 ds\right)^{\frac{1}{2}}\\
&~~ \leq L\sup_{(s,z)\in[0,t]\times[0,1]}\left(e^{-\alpha s}+e^{-\alpha s}\left(\mathbb{E}\left[\left|u(s,z)\right|^p\right]\right)^{\frac{1}{p}}\right)\left(\int_0^t e^{-2\alpha (t-s)}p_{2(t-s)}^b(x,x) ds\right)^{\frac{1}{2}}\\
&~~ \leq 2^{-\frac{1}{2}}L \sup_{(s,z)\in[0,t]\times[0,1]}\left(1+e^{-\alpha s}\left(\mathbb{E}\left[\left|u(s,z)\right|^p\right]\right)^{\frac{1}{p}}\right)\left(\int_0^{\infty} e^{-\alpha s}p_{s}^b(x,x) ds\right)^{\frac{1}{2}}\\
&~~ \leq 2^{-\frac{1}{2}}LC_{12}^{\frac{1}{2}} \left(\int_0^{\infty} e^{-\alpha s}(1+s^{-\gamma\delta}) ds\right)^{\frac{1}{2}}\sup_{(s,z)\in[0,t]\times[0,1]}\left(1+e^{-\alpha s}\left(\mathbb{E}\left[\left|u(s,z)\right|^p\right]\right)^{\frac{1}{p}}\right).
\end{align*}
We denote the gamma function by $\Gamma$ and calculate
\begin{align*}
\int_0^{\infty} e^{-\alpha s}(1+s^{-\gamma\delta}) ds&=\int_0^{\infty} e^{-\alpha s}ds +\int_0^{\infty}e^{-\alpha s}s^{-\gamma\delta}ds\\&=\frac{1}{\alpha}+\frac{1}{\alpha}\int_0^{\infty}e^{-s}\left(\frac{s}{\alpha}\right)^{-\gamma\delta}ds\\&=\frac{1+\alpha^{\gamma\delta}\Gamma(1-\gamma\delta)}{\alpha}.
\end{align*}
This implies that a constant $C_{10}'>0$ exists such that for all $\alpha>0$
\begin{align*}
\int_0^{\infty} e^{-\alpha s}(1+s^{-\gamma}) ds\leq C_{10}'\alpha^{\gamma\delta-1}.
\end{align*}
Using these esimates we get
\begin{align*}
e^{-\alpha t}\left(\mathbb{E}\left[\left|u(t,x)\right|^p\right]\right)^{\frac{1}{p}}
&\leq \left\lVert u_0\right\rVert_p\\&~+\left(\frac{L}{\alpha}+L\sqrt{2C_{12}C_{10}'\alpha^{\gamma\delta-1}p}\right)\sup_{(s,z)\in[0,t]\times[0,1]}\left(1+e^{-\alpha s}\left(\mathbb{E}\left[\left|u(s,z)\right|^p\right]\right)^{\frac{1}{p}}\right).
\end{align*}
Now, let $\alpha= C_{10}''p^{\frac{1}{1-\gamma\delta}} $, where $C_{10}''$ is a positive constant which does not depend on $p$.  Then, 
\begin{align*}
\frac{L}{\alpha}+L\sqrt{2C_{12}C_{10}'\alpha^{\gamma\delta-1}p}&=\frac{L}{C_{10}''p^{\frac{1}{1-\gamma\delta}}}+L\sqrt{2C_{11}C_{10}'}C_{10}''^{\frac{\gamma\delta}{2}-\frac{1}{2}}. 
\end{align*}
The right-hand side decreases as $C_{10}''$ increases, therefore we can choose a $C_{10}''>0$ such that 
\begin{align*}
\frac{L}{\alpha}+L\sqrt{2C_{11}C_{10}'\alpha^{\gamma\delta-1}p}\leq \frac{1}{2}.
\end{align*}
This leads to
\begin{align*}
\sup_{(s,z)\in[0,t]\times[0,1]}\left(\mathbb{E}\left[\left|u(t,x)\right|^p\right]\right)^{\frac{1}{p}}\leq \left( 2\left\lVert u_0\right\rVert_p+1\right)e^{C_{10}''p^{\frac{1}{1-\gamma\delta}} t}.
\end{align*}
We obtain for all $p\in[2,q], (t,x)\in[0,\infty)\times[0,1]$
\begin{align*}
\left(\mathbb{E}\left[\left|u(t,x)\right|^p\right]\right)^{\frac{1}{p}}\leq  \left(2\left\lVert u_0\right\rVert_q+1\right)e^{C_{10}''p^{\frac{1}{1-\gamma\delta}} t}.
\end{align*}
For $1\leq p<2$ we have for $(t,x)\in[0,\infty)\times[0,1]$
\begin{align*}
\left(\mathbb{E}\left[\left|u(t,x)\right|^p\right]\right)^{\frac{1}{p}}&\leq \left(\mathbb{E}\left[\left|u(t,x)\right|^2\right]\right)^{\frac{1}{2}}\\ &\leq
 \left(2\left\lVert u_0\right\rVert_q+1\right)e^{C_{10}''2^{\frac{1}{1-\gamma\delta}} t}\\
 &\leq
   \left(2\left\lVert u_0\right\rVert_q+1\right)e^{C_{10}''\left(\frac{2}{p}\right)^{\frac{1}{1-\gamma\delta}}p^{\frac{1}{1-\gamma\delta}} t}
\end{align*}
and obtain the assertion for all $p\in[1,q]$ by setting $C_{10}=C_{10}''2^{\frac{1}{1-\gamma\delta}}.$
\end{proof}
From the above proposition, it follows immediately for $p\in[1,q]$
\begin{align*}
\bar\gamma(p)=\limsup_{t\to\infty}\frac{1}{t}\sup_{x\in[0,1]}\log\mathbb{E}\left[|u(t,x)|^p\right]&\leq \limsup_{t\to\infty}\frac{p}{t}\log C_{11}+C_{10}p^{1+\frac{1}{1-\gamma\delta}}=C_{10}p^{1+\frac{1}{1-\gamma\delta}}.
\end{align*}
\begin{bem}
Although we have defined the mild solution only on $[0,T]$ for any $T\geq 0$, the map $t\to\mathbb{E}|u(t,x)|^p$ is continuous on $[0,\infty)$ for every $x\in[0,1]$ and every $p\geq 1$. For $t\in[0,\infty)$ choose $T=2t\wedge 1$ and use the stochastic continuity properties (see Proposition \ref{stochastic_continuity_heat}) to verify this. 
\end{bem}
In order to work out a lower bound for the second moment, we restrict ourselves to the following SPDE for $t\in[0,\infty):$
\begin{align}
\begin{split}
\frac{\partial}{\partial t}u(t,x)&=\Delta_{\mu}^Nu(t,x)+g(u(t,x))\xi(t,x),\label{spde_heat_intermittency}\\
u(0,x)&=u_0(x).
\end{split}
\end{align}
Furthermore, let the following conditions hold.

\begin{hyp}\label{hypo_heat_only_g} 
\begin{enumerate}[label=(\roman*)]
\item $u_0:[0,1]\to\mathbb{R}$ is measurable, non-negative and bounded.
\item $g:\mathbb{R}\to\mathbb{R}$ satisfies the following Lipschitz and linear growth conditions: There exists $L>0$ such that for all $(x,y)\in\mathbb{R}^2$
\begin{align*}
|g(x)-g(y)|&\leq L|x-y|,\\ |g(x)|&\leq +L(1+|x|).
\end{align*} 
\end{enumerate}
\end{hyp}

\begin{satz}\label{intermittency_heat_lower} 
Define $L_g\coloneqq \inf_{x\in[0,1]}\left|g(x)/x\right|>0.$
Let  $\inf_{x\in[0,1]}u_0(x)>0$. Then, $\gamma(2,x)\geq L_g^2$ for all $x\in[0,1]$. 
\end{satz}

\begin{proof}
From the non-negativity of $u_0$ and Proposition \ref{heat_kernel_satz}\ref{heat_kernel_7}  it follows for $x\in[0,1]$
\begin{align*}
\int_0^1 p_t^N(x,y)u_0(y)d\mu(y)\geq\inf_{y\in[0,1]}u_0(y)\int_0^1 p_t^N(x,y)d\mu(y)=\inf_{y\in[0,1]}u_0(y).
\end{align*}
By using the definition of the mild solution, Walsh's isometry, the zero-mean property of the stochastic integral and the previous estimate we get for $(t,x)\in(0,\infty)\times[0,1]$ 
\begin{align*}
\mathbb{E}\left[u(t,x)^2\right]dt
&=\left(\int_0^1p_t^N(x,y)u_0(y)d\mu(y)\right)^2
+\int_0^t\int_0^1 p_{t-s}^N(x,y)^2\mathbb{E}\left[g(u(s,y))^2\right]d\mu(y)ds\\
&~~~+\left(\int_0^1p_t^N(x,y)u_0(y)d\mu(y)\right)^2\mathbb{E}\left[\int_0^t\int_0^1p_{t-s}^N(x,y)g(u(s,y))\xi(s,y)d\mu(y)ds\right]\\
&\geq\inf_{x\in[0,1]}u_0(x)^2+\int_0^t\int_0^1L_g^2 p_{t-s}^N(x,y)^2\mathbb{E}\left[u(s,y)^2\right]d\mu(y)ds.
\end{align*} 
It holds $\varphi_1^N=1$ and $\lambda_1^N=0$ and consequently for $t\in(0,\infty), x\in[0,1]$  \[p_t(x,x)=\sum_{k\geq 1}e^{-\lambda_k^N t}\varphi_k^N(x)^2\geq e^{-\lambda_1^N t}\varphi_1^N(x)^2=1.\]
With that and $I(t)\coloneqq\inf_{x\in[0,1]}\mathbb{E}\left[u(t,x)^2\right]$ we obtain
\begin{align}
I(t)&\geq \inf_{x\in[0,1]}u_0(x)^2+\int_0^t\int_0^1 L_g^2p_{t-s}^N(x,y)^2I(s)d\mu(x)ds\notag\\
&=\inf_{x\in[0,1]}u_0(x)^2+\int_0^t\int_0^1 L_g^2p_{2(t-s)}^N(x,x)I(s)d\mu(x)ds\notag\\
&\geq \inf_{x\in[0,1]}u_0(x)^2+\int_0^tL_g^2I(s)ds.\label{p_lower_estimate}
\end{align}
It follows
\begin{align*}
I(t)&\geq \inf_{x\in[0,1]}u_0(x)^2+\int_0^tL_g^2\inf_{x\in[0,1]}u_0(x)^2ds+
\int_0^tL_g^2\int_0^s L_g^2 I(u)duds\\
&=\inf_{x\in[0,1]}u_0(x)^2+L_g^2\inf_{x\in[0,1]}u_0(x)^2t+
\int_0^tL_g^2\int_0^s L_g^2 I(u)duds
\end{align*}
and by iterating this
\begin{align*}
I(t)&\geq \inf_{x\in[0,1]}u_0(x)^2\sum_{n=0}^{\infty} \frac{L_g^{2n}t^n}{n!}=
\inf_{x\in[0,1]}u_0(x)^2e^{L_g^2t},
\end{align*}
which yields 
\begin{align*}
\liminf_{t\to\infty}\frac{1}{t}\log\mathbb{E}\left[u(t,x)^2\right]\geq\liminf_{t\to\infty}\frac{\log\inf_{x\in[0,1]}u_0(x)^2}{t}+L_g^2=L_g^2.
\end{align*}
\end{proof}
The following is the main result of this section and follows immediately from the definition of weak intermittency and Proposition \ref{intermittency_heat_lower}. 
\begin{kor}
Let $\inf_{x\in[0,1]}u_0(x)>0$. Then, the mild solution of \eqref{spde_heat_intermittency} is weakly intermittent on $[0,1]$.
\end{kor}


\begin{thebibliography}{1}


\addcontentsline{toc}{section}{References}


\bibitem{AE} P. Arzt, {\em Eigenvalues of Measure Theoretic
Laplacians on Cantor-like Sets}, Dissertation, Universität Siegen, 2014.

\bibitem{AM} P. Arzt, Measure Theoretic Trigonometric Functions, {\em Journal of Fractal Geometry}, 2(2):115-169, 2015.

\bibitem{BF} M. F. Barnsley, {\em Fractals Everywhere}, 2nd ed., Morgan Kaufmann,  1993.

\bibitem{BCS} L. Bertini, N. Cancrini, The Stochastic Heat Equation: Feynman-Kac Formula and Intermittence, {\em Journal of Statistical Physics}, 78(5-6):1377-1401, 1995.

\bibitem{BNF} E. J. Bird, S.-M. Ngai, A. Teplyaev,  Fractal Laplacians on the Unit Interval, {\em Ann. Sci. Math. Québec}, 27:135-168, 2003.

\bibitem{CJI} D. Conus, M. Joseph, D. Khoshnevisan, S. Shiu,  Intermittency and Chaos for a Non-linear Stochastic Wave Equation in Dimension 1, {\em Malliavin Calculus and Stochastic Analysis, Springer Proceedings in Mathematics \& Statistics}, 34:251-279, Springer, Boston, 2013.

\bibitem{DZS} G. Da Prato, J. Zabczyk, {\em Stochastic Equations in Infinite Dimensions}, 2nd ed., Cambridge University Press, 2014.

\bibitem{EES} D. E. Edmunds, W. D. Evans, {\em Spectral Theory and Differential Operators}, Oxford University Press, Oxford, 1987.

\bibitem{ENO} K.-J. Engel, R. Nagel, {\em One-Parameter Semigroups for Linear Evolution Equations}, 1st ed., Springer, New York, 2000.

\bibitem{ES} T. Ehnes, Stochastic Wave Equations defined by Fractal Laplacians on
Cantor-like Sets, {\em arXiv:1910.08378}, 2019.

\bibitem{FAI} W. Feller, {\em An introduction to probability theory and its applications, Vol. II}, John Wiley and Sons, New York-London-Sydney, 1966.

\bibitem{FA} U. Freiberg, Analytical properties of measure geometric Krein-Feller-operators on the real line, {\em Mathematische Nachrichten}, 260:34-47, 2003.

\bibitem{FD} U. Freiberg, Dirichlet forms on fractal subsets of the real line, {\em Real Anal. Exchange}, 30(2):589–603, 2004/05.

\bibitem{FLZ} U. Freiberg, J. Löbus, Zeros of eigenfunctions of a class of generalized second order differential operators on the Cantor set, {\em Mathematische Nachrichten}, 265:3-14, 2004.

\bibitem{FS} U. Freiberg,  Spectral asymptotics of generalized measure geometric Laplacians on Cantor like sets, {\em Forum Math.}, 17:87-104, 2005.

\bibitem{FOD} M. Fukushima, Y. Oshima, M. Takeda Dirichlet forms and symmetric Markov
processes, {\em De Gruyter Stud. Math. 19}, de Gruyter, Berlin-New York, 2011.

\bibitem{FAF} T. Fujita:   A fractional dimension, self similarity and a generalized diffusion operator, {\em Probabilistic Methods on Mathematical Physics, Proceed. of Taniguchi Int. Symp. (Katata and Kyoto,  1985)}, 83-90, Kinokuniya, 1987.

\bibitem{GJS} N. Georgiou, M. Joseph, D. Khoshnevisan, S. Shi, Semi-discrete semi-linear parabolic SPDEs, {\em Ann. Appl. Probab.}, 25(5):2959–3006, 2015.

\bibitem{HYE} B. Hambly, W. Yang, Existence and space-time regularity for stochastic heat equations on p.c.f. fractals, {\em Electron. J. Probab.}, 23(22):1-30, 2018.

\bibitem{HYC} B. Hambly, W. Yang, Continuous random field solutions to parabolic SPDEs on p.c.f. fractals, {\em arXiv:1709.00916v2}, 2018.

\bibitem{HHS} Y. Hu, J. Huang, D. Nualart, S. Tindel, Stochastic heat equations with general multiplicative Gaussian noises: Hölder continuity and intermittency, {\em Electron. J. Probab}, 20(55):1-50, 2015.

\bibitem{HF} J. E. Hutchinson, Fractals and Self Similarity, {\em Indiana Univ. Math. J}, 30(5):713-747, 1981.

\bibitem{IKD} K. Itô, H. P. Jr. McKean, {\em Diffusion Processes and their Sample Paths}, Springer-Verlag, Berlin-Heidelberg-New York, 1965.

\bibitem{KAS} D. Khoshnevisan, {\em Analysis of Stochastic Partial Differential Equations}, American Mathematical Society, 2014.

\bibitem{KKI} D. Khoshnevisan, K. Kim, Y. Xiao, Intermittency and Multifractality: A case study via parabolic stochastic PDEs, {\em Ann. Probab.}, 45(6A):3697-3751, 2017.

\bibitem{KA} J. Kigami, Analysis on Fractals, {\em Cambridge Tracts in Mathematics 143}, Cambridge University Press, Cambridge, 2001.

\bibitem{KW} A. Klenke, {\em Wahrscheinlichkeitstheorie}, Springer, Berlin, 2013.

\bibitem{KS} U. Küchler, Some Asymptotic Properties of the Transition Densities of One-Dimensional Quasidiffusions, {\em Publ. RIMS, Kyoto Univ.}, 16:245–268, 1980.

\bibitem{KO} U. Küchler, On sojourn times, excursions and spectral measures connected with quasidiffusions, {\em J. Math. Kyoto Univ.}, 26(3):403-421, 1986.

\bibitem{LRS} W. Liu, M. Röckner, {\em Stochastic Partial Differential Equations: An Introduction}, 1st ed., Springer,  2015.

\bibitem{LG} J.-U. Löbus, Generalized second order differential operators, {\em Mathematische Nachrichten}, 152:229-245, 1991.

\bibitem{WI} J. Walsh,  An Introduction to Stochastic Partial Differential Equations, In {\em École d'Été de Probabilités de Saint Flour}, XIV-1984, 1180:265–439, Springer, Berlin, 1986.
\end{thebibliography}
\end{document}